\documentclass[a4,psfig,english,12pt,epsf,portrait]{article}
\usepackage{amsmath,amsfonts,latexsym,amscd,hyperref}
\usepackage[english]{babel}
\usepackage{epsfig}
\usepackage[T1]{fontenc}
\usepackage[ansinew]{inputenc}  
\usepackage{graphics}
\usepackage{graphicx}
\usepackage{tikz}
\usepackage{amsmath}
\usepackage{amsthm}
\usepackage{amssymb}
\usepackage{enumerate}
\usepackage{multicol}
\usepackage{hyperref}

\newtheorem{theorem}{Theorem}[section]
\newtheorem{remark}{Remark}[section]
\newtheorem{proposition}[theorem]{Proposition}
\newtheorem{definition}{Definition}[section]
\newtheorem{lemma}[theorem]{Lemma}

\numberwithin{equation}{section}

\def\R{{\mathbb {R}}}

\DeclareMathOperator*{\supp}{supp}

\usepackage{color}

\def\d{\displaystyle}
\def\e{{\varepsilon}}

\def\R{\mathbb{R}}

\title{Critical exponent for the one-dimensional wave equation with a space-dependent scale invariant damping and time derivative nonlinearity}

\author{
Ahmad Z. Fino\\
{\it \small 
College of Engineering and Technology, American University of the Middle East, Kuwait}\\
Mohamed  Ali  Hamza\\
{\it \small 
Imam Abdulrahman Bin Faisal University,
 Dammam, 34212, Saudi Arabia}\\
}

\date{}

\begin{document}
	\maketitle

\begin{abstract}
We investigate in this paper the Cauchy problem of the  one-dimensional wave equation with space-dependent damping of the form 
$\mu_0(1+x^2)^{-1/2}$, where $\mu_0>0$, and time derivative nonlinearity. 
 We establish global existence of mild solutions for small data compactly supported by employing energy estimates within suitable Sobolev spaces of the associated homogeneous problem.
  Furthermore, we derive a blow-up result under some positive initial data  by employing the test function method.  
This shows that the critical exponent  is  given by $p_G(1+\mu_0)=1+2/\mu_0$, when $\mu_0\in (0,1]$, where $p_G$ is the Glassey exponent.
To the best of our knowledge, this  constitutes  the first identification of the critical exponent range for this class of equations. As by product, we extend the global existence result to a more general class of space/time nonlinearities of the form $f(\partial_tu,\partial_x u)=|\partial_x u|^{q}$ or $f(\partial_tu,\partial_x u)=|\partial_tu|^{p}|\partial_x u|^{q}$,  with $p,q>1$.
\end{abstract}

\medskip




\section{Introduction}

In this work, we consider the Cauchy problem for  the wave equations with 
critical space-dependent damping
and  power-nonlinearity of derivative type
\begin{equation}
\label{NLW}
\left\{
\begin{array}{ll}
 \displaystyle\partial_{t}^2u - \partial_x^2 u +V(x)\partial_tu= f(\partial_tu), &\quad  x\in \mathbb{R},\, t>0, \\
u(x,0)= u_0(x), \quad \partial_tu(x,0)= u_1(x), &\quad x\in\mathbb{R},
\end{array}
\right.
\end{equation}
where $p>1$, $V(x)=\mu_0(1+x^2)^{-1/2}$, $\mu_0>0$, $f(\partial_tu)=|\partial_tu|^p $ or $|\partial_tu|^{p-1}\partial_tu$, and 
$(u_0,u_1)$ are compactly supported   initial data  in the following Sobolev space 
\begin{equation}\label{initialdata}
 (u_0,u_1)\in H^{2}(\mathbb{R})\times H^{1}(\mathbb{R}).
  \end{equation}
Such problems appear in models for wave propagation in a nonhomogeneous gas with damping, where the space-dependent coefficients represent either friction coefficients or potential (see \cite{Ikawa}).\\

In this paper, we investigate both the global existence of small data solutions and the finite-time blow-up of solutions to the Cauchy problem for the nonlinear wave equation \eqref{NLW}.  These results allow us to determine the critical exponent for $\mu_0\in (0,1]$.
\subsection{Prior results}
Before going to the main result, we shall review prior studies on the global existence/nonexistence of solutions to some nonlinear damped wave equations with various type of potentials and nonlinearities.

For the Cauchy problem of the semilinear damped wave equation with power nonlinearity and when the variable coefficients are missing, 
\begin{equation}\label{classicalcase}
\partial_{tt}u-\Delta u + \partial_t u=|u|^p,\qquad t>0,\,\,x\in\mathbb{R}^n,
\end{equation}
it has been conjectured that \eqref{classicalcase} has the diffuse structure as $t\rightarrow\infty$ (see e.g. \cite{Belloutfriedman}), i.e. the solutions of the damped equation seem to behave more like solutions of the corresponding diffusion/heat equation $\partial_t u-\Delta u =|u|^p$ at large times. This suggests that problem \eqref{classicalcase} should have $p_F(n):=1+{2}/{n}$ as critical exponent which is called the Fujita exponent named after the pioneering work by Fujita \cite{Fujita}. Indeed, Li and Zhou \cite{LiZhou} proved that if $n \leq 2$, $1 < p \leq 1 + 2/n$ and the data are positive on average, then the local
solution of \eqref{classicalcase} must blows up in a finite time. Moreover,
Todorova and Yordanov \cite{Todorovayordanov} and Zhang \cite{Zhang2001} showed that the critical exponent of \eqref{classicalcase} is the predicted one for any $n\geq1$. We mention that Todorova and Yordanov \cite{Todorovayordanov} proved the global existence by assuming that the data are compactly supported.

The study of 
 the  analogous   problem  of \eqref{classicalcase}
with space-dependent damping, namely
\begin{equation}\label{D1}
\partial_{t}^2u - \Delta u +\frac{\mu_0}
 {(1+|x|^2)^{\alpha/2}} \partial_tu= |u|^p \quad\text{in}\,\, \R^n\times(0, \infty),
\end{equation} 
 has been the subject of
several works in the literature.  First,
for $\alpha \in (0,1)$ the 
critical exponent is given by the Fujita-type exponent $p_F(n-\alpha)=1+2/(n-\alpha)$ (see Ikehata,    Todorova and Yordanov \cite{ITY}).
Later,    Nishihara,  Sobajima and  Wakasugi  proved in \cite{NSW} 
 that the critical exponent remains the Fujita-type exponent,   for $\alpha<  0$.
 The situation is completely different 
when $\alpha>1$, (scatering case). More precisely,
Lai and Tu  \cite{LT3}
showed that the critical exponent of \eqref{D1} is given by $p=p_S(n)$, where $p_S(n)$ stands for the 
so-called  Strauss exponent which is the critical exponent of the
 well-known equation  $\partial_{t}^2u -  \Delta u = |u|^p$. 
This explains the scattering phenomena.
Finally, when $\alpha=1$, the problem \eqref{D1} has a 
scale-invariance
and the critical exponent seems  $\max (p_F(n-1), p_S(n+\mu_0))$.
 In fact, Li \cite{Li} proved the  expected result in the  case $\mu_0\ge n$.
Also,  the blow-up part of this conjecture  is  known  (for small
values of $\mu_0$).
For the singular damping term $a(x)=a_0|x|^{-1}$,  global existence has been investigated in the case of an exterior domain in  a recent paper by Sobajima \cite{Sobajima5}.  In the case of whole space,  a blow-up result and lifespan estimate were also established by Ikeda and Sobajima \cite{Ikeda-Sobajima2021}.

In the context of  the general case where the damping coefficient  is  critical and is a function of both time and space
namely     $V(x,t)=\mu_0(1+|x|^2)^{-\alpha/2}(1+t)^{-\beta},$  the situation is much less understood. 
We  refer the reader to \cite{NSW,W0}.

Now, we turn back to the classical semilinear wave equation with time derivative nonlinearity without  damping term, namely    
\begin{equation}\label{B00}
 \partial^2_{t}u-\Delta u=|\partial_tu|^p,
\quad \mbox{in}\ \R^n\times(0,\infty),
\end{equation}
 for which the critical exponent is determined by the Glassey exponent 
\begin{equation}\label{Glassey}
p_G=p_G(n):=1+\frac{2}{n-1},
\end{equation}
see e.g. \cite{Agemi,Hidano1,Hidano2,John1,Rammaha,Sideris,Zhou1}.

For the  case with time dependent damping in the scale-invariant case, namely the following equation
\begin{equation}\label{B}
 \partial^2_{t}u-\Delta u+\frac{\mu_0}{1+t}\partial_tu=|\partial_tu|^p,
\quad \mbox{in}\ \R^n\times(0,\infty),
\end{equation}
the blow-up region for $p\in (1, p_G(n+\mu_0)]$ and lifespan estimates for solutions of \eqref{B} are established in \cite{Our2},  improving upon the earlier results in \cite{Palmieri, LT2}.  On the one hand, research on blow-up phenomena and lifespan estimates has centered on the scale-invariant model equation \eqref{B}, particularly the failure of small-data global solutions when \eqref{B} incorporates additional terms or when it is extended, via coupling, to wave-like systems. For example, \cite{BHHT,HH-apam, HHP1, HHP2} analyze Tricomi-type models to show blow-up, while \cite{HH-jaac,Our9} demonstrate that mass terms have no impact even with scale-invariant damping. Meanwhile, \cite{Our,Our2} reveal how mixed nonlinearities drive  singularity formation. 
On the other hand, our recent work in \cite{HFglobal} demonstrates global existence for the one-dimensional case, leading to the identification of   $p_G(1+\mu_0)$  as the critical exponent for  $\mu_0\in (0,2]$. 

\par

Replacing the damping term in equation  \eqref{B} with $1/(1+t)^{\alpha}\partial_tu$, for  $\alpha>1$, leads to dynamics characterized by the scattering regime.
Due to the weakened damping effect for large time, its influence on the long-term dynamics is expected to be negligible. Consequently, the critical exponent is expected to coincide with the Glassey exponent, defined in \eqref{Glassey}. While Lai and Takamura in \cite{LT2} have demonstrated the blow-up part, the question of global existence remains an open problem.

\par

In the same way, the analogous  nonlinear problem \eqref{NLW}  in higher dimension with
$V(x)$ being  changed by
$\mu_0(1+|x|^2)^{-\alpha/2},$  and $\alpha>1$,
namely
\begin{equation}\label{D}
\partial_{t}^2u - \Delta u +\frac{\mu_0}
 {(1+|x|^2)^{\alpha/2}}\partial_tu= |\partial_tu|^p  \quad\text{in}\,\, \R^n\times(0, \infty),
\end{equation}
 which corresponds to the scattering case,
 it is  reasonable to  expect that
the blow-up region of the solution of \eqref{D}  is similar to the  case of the  pure   wave equation and  the scattering damping term has no
influence in the dynamics. 
The predicted blow-up result was obtained in the case where $\alpha>2$,  by Lai and Tu  \cite{LT3}.
However, up to our knowledge, there is no result in the case where $\alpha \in (1,2].$
The condition $\alpha>2$ seems to be technical. Therefore, the result stated in \cite{LT3} should be extended to the case $\alpha>1$. It is important to note that the question of global existence remains an open problem.

Recently, in \cite{HF}, we proved that the blow-up region of the present problem \eqref{NLW}  is given by $p \in (1, p_G(n+\mu_0)]$  in higher dimensions $n\geq2$. However, up to our knowledge, there are no results on the blow-up in finite time in the one dimensional case as well as the small data global existence of solutions is still an open problem in any dimension.

This paper aims to address and resolve the critical exponent problem associated with equation \eqref{NLW}.  As previously explained for the problem  \eqref{B} in the one-dimensional case, where damping induces in the critical exponent a shift in the dimensional parameter of magnitude  $\mu_0$, for $\mu_0\in (0,2]$, we naturally expect 
that the same phenomenon holds for the problem \eqref{NLW} at least for small values of $\mu_0$.  In other words,  the predicted  critical exponent is  $p_c(1,\mu_0)=p_G(1+\mu_0)$, at least for small values of $\mu_0$.
 We will achieve this in two parts.
\begin{itemize}
\item Part I: Establish global existence results for the problem \eqref{NLW}, it is essential to analyze the associated linear problem and utilize decay rates in appropriate Sobolev spaces. That question 
was done by Ikehata-Todorova-Yordanov in \cite{Ikehata}. Indeed,  they showed that, in the case where the initial data  are compactly supported, the  solution satisfies the same kind of  energy estimates, namely
$\|(\nabla u(t), \partial_tu(t))\|_{L^2}\le C    (1+t)^{-\gamma},$ for some  $\gamma>0$,  and where the constant $C$ depends on the support of the initial data. 
We would like to draw the attention of the reader to the fact that a better understanding of the dependence of the constant related to the initial data is crucial for establishing a global existence result for the corresponding nonlinear problem. In this context,   in recent work in  \cite{Sobajima5}, the author successfully established a global existence result for the nonlinear problem with space-dependent damping in an exterior domain. This achievement was made  by employing weighted energy estimates for the associated linear problem.
In the present paper,   we  start by studying  the associated homogeneous problem,  by using  energy estimates
 for compactly supported initial data. Subsequently,  we address the nonlinearity using the fixed-point theorem.
Therefore, the following upper bound for the critical exponent of problem \eqref{NLW} is obtained:
\begin{equation}\label{critical1}
p_c(1,\mu_0)\le 1+\frac{2}{\alpha_0}, 
\end{equation}
where
\begin{equation}\label{alpha00}
\alpha_0:=\min (1,\mu_0).
\end{equation}

\item Part II: 
For initial data with compact support, we derive a blow-up result for the nonlinear wave equation \eqref{NLW},  by employing the test function method.  Specifically, we employ a test function constructed as the product of a  cut-off function and an explicit solution to the adjoint equation associated with the linear part of \eqref{NLW}. This approach establishes a lower bound for the blow-up region for the solution of equation \eqref{NLW}.
Hence,  the following lower bound for the critical exponent of problem \eqref{NLW} is established:
\begin{equation}\label{critical2}
p_c(1,\mu_0)\ge 1+\frac{2}{\mu_0}.
\end{equation} 
\end{itemize}

\par

By combining the results from Parts I and II, we easily deduce that     
\begin{equation}\label{critical11}
p_c(1,\mu_0)= 1+\frac{2}{\mu_0}, \quad\textrm{if}\,\, \mu_0 \in(0,1].
\end{equation}
\par

At the end of this subsection we prepare notation and several definitions used throughout this paper. 
We denote by $C$ a positive constant, which may change from line to line.
$L^p = L^p(\mathbb{R})$
stands for the usual Lebesgue space,
and
$H^{k} = H^{k}(\mathbb{R})$
for $k \in \mathbb{Z}_{\ge 0}$
is the Sobolev space defined by:
\begin{align*}
	H^{k}(\mathbb{R}) = \left\{ f \in L^2(\mathbb{R}) ;
			\| f \|_{H^{k}} = \sum_{\ell = 0}^k \|  f^{(\ell)}  \|_{L^2} < \infty \right\}.
\end{align*}

\subsection{Main Result}\label{subsec3}
The purpose of this subsection is to state our main results. Recall that $f(\partial_tu)=|\partial_tu|^p $ or $|\partial_tu|^{p-1}\partial_tu$. We start by giving the definition of mild solution of the system (\ref{NLW}). 
\begin{definition}(Mild solution)\\ 
Assume that $(u_0,u_1)\in H^{2}(\mathbb{R})\times H^{1}(\mathbb{R})$. We say that a function $u$ is a
mild solution of (\ref{NLW}) if $u\in \mathcal{C}^1([0,T],H^{1}(\mathbb{R}))$ is subject to the initial conditions $u(0)=u_0$, $\partial_t u(0)=u_1$, and satisfies
the integral equation
 \begin{equation}\label{mildsolution}
u(t,x)=R(t)(u_0,u_1)+\int_0^tS(t,s)f(\partial_tu)\,ds=:u^{lin}(t)+u^{nl}(t)
\end{equation} 
in the sense of $H^{1}(\mathbb{R})$, where the operators $R(t)$ and $S(t,s)$ are defined in Subsection \ref{subs2.3} below. If the time $T>0$ can be chosen arbitrary, then $u$ is called a global-in-time mild solution.
\end{definition}
For later use, we define the following functional space, which will be used in the main results:
$$E:=\mathcal{C}([0,\infty),H^2(\mathbb{R}))\cap    \mathcal{C}^1([0,\infty),H^1(\mathbb{R}))\cap \mathcal{C}^2([0,\infty),L^2(\mathbb{R})).$$
\begin{theorem}$(\hbox{Small data global existence})$\label{globalexistence1}\\ 
Assume that  $u_0 \in H^2(\R)$ and $u_1 \in H^1(\R)$ which are compactly supported on  $(-R_0,R_0)$, with $R_0>0$. 
If $p>1+{2}/{\alpha_0}$, where $\alpha_0:=\min (\mu_0,1)$, then there exists a positive constant $0<\varepsilon_0\ll1$, small enough, such that for any initial data satisfying 
$$\|u_0\|_{H^2}+\|u_1\|_{H^1}\leq\varepsilon_0,$$
 problem (\ref{NLW}) possesses a uniquely global mild solution  $u\in E$. Moreover, the solution satisfies the following estimate
$$(1+t)^{-1}\|u(t)\|_{L^2}+
\|\partial_x u(t)\|_{H^1}+\| \partial_t u(t)\|_{H^1}+\| \partial^2_t u(t)\|_{L^2}\leq C\,(1+t)^{-\frac{\alpha}{2}},$$
where $\alpha$ is defined by \eqref{alpha}.
\end{theorem}

To state the blow-up result, we first examine the relationship between the size of the initial data and the lifespan of solutions. For this reason, it is appropriate to scale the initial conditions by a small positive parameter $ \varepsilon $. Specifically, we consider the following problem
\begin{equation}
\label{NLWb}
\left\{
\begin{array}{ll}
 \displaystyle\partial_{t}^2u - \partial_x^2 u +V(x)\partial_tu= f(\partial_tu), &\quad  x\in \mathbb{R},\, t>0, \\
u(x,0)= \varepsilon u_0(x), \quad \partial_tu(x,0)= \varepsilon u_1(x), &\quad x\in\mathbb{R},
\end{array}
\right.
\end{equation}
To formulate our blow-up results, we now present the weak formulation of equation \eqref{NLWb} in the appropriate energy space.
\begin{definition}\label{def1}
Let  $u_0\in H^1(\mathbb{R})$, $u_1\in L^2(\mathbb{R})$  and $T>0$.  Let $u$ be such that
$u\in \mathcal{C}([0,T),H^1(\R))\cap \mathcal{C}^1([0,T),L^2(\R)) \
\text{and} \ \partial_tu \in L^p_{loc}((0,T)\times \R),$
verifies,  for any $\varphi \in \mathcal{C}^1_0\left([0, T)\times \R\right) \cap \mathcal{C}^{\infty}\left((0, T)\times \R\right)$, the following identity:
\begin{eqnarray}\label{weaksol}
&{}& \e \int_{\R}u_1(x)\varphi( x,0)dx+\int_0^T\int_{\R}f(\partial_tu)\varphi(x, t) \, dxdt \\
&{}&=- \int_0^T\int_{\R} \partial_tu(t, x)\partial_t\varphi( x,t) \, dxdt
+ \int_0^T\int_{\R}
\ \partial_x u(t, x)\partial_x\varphi(x, t) \, dxdt + \int_0^T\int_{\R} V(x) \partial_tu(t, x)\varphi( x,t) \, dxdt,\nonumber
\end{eqnarray}
and the condition $u(x,0)=\varepsilon u_0(x)$,  is satisfied. Then, $u$ is called an {\bf energy solution} of (\ref{NLWb}) on $[0,T)$. We denote the lifespan for the energy solution by
$$T_{\e}(u_0,u_1):=\sup\{T\in(0,\infty];\,\,\hbox{there exists a unique energy solution $u$ of \eqref{NLWb}}\}.$$
Moreover, if $T>0$ can be arbitrary chosen, i.e. $T_{\e}(u_0,u_1)=\infty$, then $u$ is called a global energy solution of \eqref{NLWb}.
\end{definition}
\begin{theorem}$(\hbox{Blow-up})$\label{th2}\\
Assume that  $u_0 \in H^2(\R)$ and $u_1 \in H^1(\R)$ which are compactly supported functions on  $(-R_0,R_0)$, with $R_0>0$, and satisfy
  \begin{equation}\label{C0}
   \int_{\R} \big( u_0''(x)+ u_1(x)\big)\phi(x) dx >0,
  \end{equation}
  where $\phi(x)$  is  a solution of the elliptic problem \eqref{phi}. If $1<p\le 1+2/{\mu_0},$ then the energy solution $u$ of \eqref{NLWb} with compact support
\begin{equation}\label{suppcond}
\supp u \in	\left\{(x,t) \in \R  \times   [0,T) \colon |x| \le R_0+ t\right\}
\end{equation}
blows-up in finite time. More precisely, there exists a constant $\e_0=\e_0(u_0, u_1, \mu_0, p,R_0)>0$
such that the lifespan $T_\e$ verifies
\begin{equation}\label{T-epss}
T_\e \leq
\d \left\{
\begin{array}{ll}
 C \, \e^{-\frac{2(p-1)}{2-\mu_0 (p-1)}}
 &
 \ \text{if}\, \
 1<p<1+\frac2{\mu_0}, \vspace{.1cm}
 \\
 \exp\left(C\e^{-(p-1)}\right)
&
 \ \text{if}\, \ p=1+\frac2{\mu_0},
\end{array}
\right.
\end{equation}
for $0<\e\le\e_0$ and some constant $C$ independent of $\e$.
\end{theorem}

\begin{remark}
We emphasize, by Theorems \ref{globalexistence1} and \ref{th2}, that when $\mu_0 \in (0,1]$
the  critical exponent in the one-dimensional case is  given by 
$$p_c(1,\mu_0)= 1+\frac{2}{\mu_0}.$$
Let us mention that, When  $\mu_0 >1$, the critical exponent $p_c(1,\mu_0)$ is known to lie within the interval $ [1+2/\mu_0,3]$.  However, the precise value of $p_c(1,\mu_0)$ in this case remains an open question.
\end{remark}

\par

We conclude by considering the generalized Cauchy problem given by
\begin{equation}
\label{General}
\left\{
\begin{array}{ll}
 \displaystyle\partial_{t}^2u - \partial_x^2 u +V(x)\partial_tu= f(\partial_tu,\partial_x u), &\quad  x\in \mathbb{R},\, t>0, \\
u(x,0)= u_0(x), \quad \partial_tu(x,0)= u_1(x), &\quad x\in\mathbb{R},
\end{array}
\right.
\end{equation}
where $V(x)=\mu_0(1+x^2)^{-1/2}$, $\mu_0>0$, and the nonlinearity is of the form
$$f(\partial_tu,\partial_x u)=|\partial_tu|^{\beta}|\partial_x u|^{\gamma},$$
with either $(\beta,\gamma)=(0,q)$ or $(\beta,\gamma)=(p,q)$, $p,q>1$.
\begin{theorem}$(\hbox{Small data global existence: General Case})$\label{globalexistence2}\\ 
Assume that  $u_0 \in H^2(\R)$ and $u_1 \in H^1(\R)$ which are compactly supported on  $(-R_0,R_0)$, with $R_0>0$. Assume that the nonlinear exponents \( \beta, \gamma > 0 \) satisfy one of the following conditions:
\begin{itemize}
    \item[$(\mathrm{i})$] \( (\beta, \gamma) = (0, q) \) with \( q > 1 + 2/\alpha_0 \),
    \item[$(\mathrm{ii})$] \( (\beta, \gamma) = (p, q) \) with \( p , q > 1 \), and \( p + q > 1 + 2/\alpha_0 \),
\end{itemize}
where $\alpha_0:=\min (\mu_0,1)$. Then there exists a positive constant $0<\varepsilon_0\ll1$, small enough, such that for any initial data satisfying 
$$\|u_0\|_{H^2}+\|u_1\|_{H^1}\leq\varepsilon_0,$$
problem \eqref{General} possesses a uniquely global mild solution $u\in E$. Moreover, the solution satisfies the decay estimate
$$(1+t)^{-1}\|u(t)\|_{L^2}+
\|\partial_x u(t)\|_{H^1}+\| \partial_t u(t)\|_{H^1}+\| \partial^2_t u(t)\|_{L^2}\leq C\,(1+t)^{-\frac{\alpha}{2}},$$
where $\alpha$ is defined by \eqref{alpha}.
\end{theorem}

\par

This paper is organized as follows:
 Section 2 is devoted to the energy estimates of the solution of the corresponding linear equation of \eqref{NLW}.  Section \ref{sec3} is dedicated to the proof of the small data global existence result (Theorem \ref{globalexistence1}) while the proof of the blow-up result (Theorem \ref{th2})  is given in Section \ref{sec5}. 
 Finally, Section \ref{sec4} generalizes the global existence results to a  class of nonlinearities (Theorem \ref{globalexistence2})
 

\section{Homogeneous equation}\label{sec2}
Having in mind that the asymptotic profile of solutions to the linear part of the equation influences the
critical exponent for the problem, we consider the following homogeneous problem
\begin{equation}\label{1}
\begin{cases}
\partial_{t}^2v-\partial_x^2 v +V(x) \partial_{t}v=0, &x\in \mathbb{R},\,t>s_0,\\
v(x,s_0)=  g(x,s_0),\, \partial_{t}v(x,s_0)=  h(x,s_0),&x\in \mathbb{R},\\
\end{cases}
\end{equation}
where $s_0\geq0$ is a fixed time, and the potential is given by $V(x)=\mu_0(1+x^2)^{-1/2}$ with $\mu_0>0$. In addition, we assume that the initial data $(g(s_0),h(s_0))$ are compactly supported in the interval $(-R_0-s_0,R_0+s_0)$, for some $R_0>0$, and that $(g(s_0),h(s_0))\in H^2(\mathbb{R})\times H^1(\mathbb{R})$.\\
We begin by providing the definition of a strong solution to \eqref{1}. 
\begin{definition}[Strong solution]${}$\\
Let $(g(s_0),h(s_0))\in H^2(\mathbb{R})\times H^1(\mathbb{R})$. A function $v$ is said to be a strong solution to \eqref{1} if
	$$	v\in  \mathcal{C}\left([s_0,\infty),H^2(\mathbb{R})\right)\cap \mathcal{C}^{1}\left([s_0,\infty),H^1(\mathbb{R})\right)\cap \mathcal{C}^{2}\left([s_0,\infty),L^2(\mathbb{R})\right)
,$$
and $u$ has initial data $v(\cdotp,s_0)=g(\cdotp,s_0)$, $\partial_tv(\cdotp,s_0)=h(\cdotp,s_0)$ and satisfies
the equation in \eqref{1} in the sense of $L^2(\mathbb{R})$.
\end{definition}
By making the change of variable $t\mapsto t-s_0$ on Theorem 2.27  in \cite{Ikawa}, we have
\begin{theorem}\label{existencetheorem}${}$\\
	Let  $m\in\mathbb{N}$. For each $(g(s_0),h(s_0))\in H^{m+2}(\mathbb{R})\times H^{m+1}(\mathbb{R})$, there exists a unique strong solution $v$ to \eqref{1} such that $\supp v(t)\subset\{x\in \mathbb{R};\,\,|x|\leq t+R_0\}$ and	
	$$v\in\bigcap_{j=0}^{m+2}  \mathcal{C}^{m+2-j}\left([s_0,\infty),H^j(\mathbb{R})\right).$$
	\end{theorem}

This section is dedicated to the statement and proof of Proposition \ref{prop2.1} and is organized into three subsections:
\begin{itemize}
\item In the first one, we present the two classical energy estimates, obtained by multiplying the equation in \eqref{1} by
 $\partial_tv$ and $v$, respectively. These computations make use of the fact that the initial data $(g(s_0),h(s_0))$ are compactly supported in the interval $(-R_0-s_0,R_0+s_0)$. The combination of these estimates leads to Lemma \ref{LEE}.
 \item The second subsection is devoted to refining these energy estimates, yielding sharper decay bounds.
  \item Subsequently, we combine the previous energy estimates to establish the proof of Proposition \ref{prop2.1}.
\end{itemize}

 \subsection{A bound of the $H^1\times L^2$ norm of $(v,\partial_tv)$}
This subsection presents Lemma \ref{LEE}, which based on the construction of a Lyapunov functional. We begin by introducing three crucial functionals as follows:
\begin{eqnarray}
E_0(v(t),\partial_tv(t))&:=&\frac12\displaystyle\int_{\R}\big ((\partial_xv(t))^2+
(\partial_tv(t))^2 \big ) {\mathrm{d}}x,\label{En}\qquad
\\
I(v(t),\partial_tv(t),t)&:=&\displaystyle\int_{\R}\big(v(t)\partial_tv(t)+ \frac{1}{2(t+1)}(v(t))^2+ \frac{V(x)}{2}(v(t))^2
 \big ) {\mathrm{d}}x,\quad \label{F2}\\
   F_{A}(v(t),\partial_tv(t),t)&:=&AE_0(v(t),\partial_tv(t))+\frac{1}{2(t+1)}I(v(t),\partial_tv(t),t),\quad\label{F4}
\end{eqnarray}
where $A>0$ is a constant to be specified later. By evaluating the time derivative of $E_{0}(v(t),\partial_tv(t),t)$, $I(v(t),\partial_tv(t),t)$ and $F_{A}(v(t),\partial_tv(t),t)$, we establish the following result.
\begin{lemma} \label{LEE} 
 Assume
$(g(s_0),h(s_0))\in H^{2}(\mathbb{R})\times H^{1}(\mathbb{R})$ that are compactly supported 
on $(-R_0-s_0,R_0+s_0)$, with $R_0>0$, then there exists a constant $C_0=C_0(R_0,\mu_0)>0$ such that the strong solution $v$ of \eqref{1} satisfies, 
for all   $t\ge s_0$,
\begin{eqnarray*}
	&{}&\frac{\|v(t)\|^2_{L^2}}{(1+t)^{2}}+\frac{\|\sqrt{V}v(t)\|^2_{L^2}}{1+t}
+	\|\partial_tv(t)\|^2_{L^2}+\|\partial_x v(t)\|^2_{L^2}\\
&{}&\leq C_0
\left( 
	\frac{\|g(s_0)\|^2_{L^2}}{(1+s_0)^2}+\frac{\|\sqrt{V}g(s_0)\|^2_{L^2}}{1+s_0}
+	\|h(s_0)\|^2_{L^2}+\|\partial_xg(s_0)\|^2_{L^2}\right).
\end{eqnarray*}
\end{lemma}
\begin{proof}	
Since
$(g(s_0),h(s_0))\in H^{2}(\mathbb{R})\times H^{1}(\mathbb{R})$, Theorem \ref{existencetheorem} guarantees that the solution $v$ to problem \eqref{1} satisfies 
$v\in \bigcap_{j=0}^{2}  \mathcal{C}^{j}\left([s_0,\infty),H^{2-j}(\mathbb{R})\right)$. As a consequence, the energy functional $E_0(v(t),\partial_tv(t))$ is a differentiable with respect to time. Moreover, by multiplying the equation in  \eqref{1} by $\partial_tv(t)$ and  integrating by parts, we get
\begin{equation}\label{E00}
 \frac{d}{dt}E_0(v(t),\partial_tv(t))=-\displaystyle\int_{\R}
V(x)(\partial_tv(t))^2  {\mathrm{d}}x,\qquad\hbox{for all}\,\, t>s_0.
\end{equation}
Similarly, the functional $I(v(t),\partial_tv(t),t)$ is differentiable with respect to time. Proceeding in the same manner, by multiplying the equation in \eqref{1} by $v(t)$ and applying integration by parts, we obtain, for all $t>s_0$,
\begin{eqnarray}\label{F1}
 \frac{d}{dt}I(v(t),\partial_tv(t),t)&=&\displaystyle\int_{\R}
(\partial_tv(t))^2  {\mathrm{d}}x
-\displaystyle\int_{\R}
(\partial_x v(t))^2  {\mathrm{d}}x\\
&&-\frac{1}{2(1+t)^2}\displaystyle\int_{\R}
(v(t))^2  {\mathrm{d}}x+\frac{1}{1+t}\displaystyle\int_{\R}
v(t)\partial_tv(t)  {\mathrm{d}}x.\nonumber
\end{eqnarray}
On the other hand, by making use of identities \eqref{E00}, \eqref{F1}, along with the definition of  $F_A(v(t),\partial_tv(t),t)$ given by \eqref{F4}, we  infer that
\begin{eqnarray*}
 \frac{d}{dt}F_{A}(v(t),\partial_tv(t),t)&=&- A\displaystyle\int_{\R}
V(x)(\partial_tv(t))^2  {\mathrm{d}}x+\frac{1}{2(1+t)}\displaystyle\int_{\R}
(\partial_tv(t))^2  {\mathrm{d}}x
-\frac{1}{2(1+t)}\displaystyle\int_{\R}
(\partial_x v(t))^2  {\mathrm{d}}x\\
&&-\frac{1}{2(1+t)^3}\displaystyle\int_{\R}
(v(t))^2  {\mathrm{d}}x
-\frac{1}{4(1+t)^2 }\displaystyle\int_{\R}
V(x)(v(t))^2  {\mathrm{d}}x\\
&\le &- A\displaystyle\int_{\R}
V(x)(\partial_tv(t))^2  {\mathrm{d}}x+\frac{1}{2(1+t)}\displaystyle\int_{\R}
(\partial_tv(t))^2  {\mathrm{d}}x.
\end{eqnarray*}
Using the fact that the support of $v$ is contained within the region $|x|\leq R_0+t$, we have
$$V(x)\geq \frac{\mu_0}{\sqrt{1+(R_0+t)^2}}\geq\frac{\mu_0}{R_0+1+t}\geq\frac{\mu_0}{(1+R_0)(1+t)},$$
for any $t\geq s_0$. Therefore, for all $A\geq (1+R_0)/2\mu_0$, it follows that
\begin{eqnarray*}
 \frac{d}{dt}F_{A}(v(t),\partial_tv(t),t)&\le &- \left(\frac{A \mu_0}{1+R_0}-\frac{1}{2}\right)\int_{\R}\frac{(\partial_t v(t))^2}{1+t}{\mathrm{d}}x\leq 0.
\end{eqnarray*}
Hence, by integrating over the interval $[s_0,t]$  we arrive at
\begin{equation}\label{F19B}
F_{A}(v(t),\partial_tv(t),t) \leq   F_{A}(v(s_0),\partial_tv(s_0),s_0),\quad\hbox{for  any}\,\,  t\ge s_0.
\end{equation}
Moreover, by applying Young's inequality $|ab|\leq \frac{1}{4}a^2+b^2$ with $a=v/(1+t)$, $b=\partial_tv$, we obtain
$$\frac{1}{2}\displaystyle\int_{\R}\frac{v(t)\partial_t v(t)}{1+t}{\mathrm{d}}x\geq -\frac{1}{8}\int_{\R}\frac{(v(t))^2}{(1+t)^2}{\mathrm{d}}x-\frac{1}{2}\int_{\R}(\partial_t v(t))^2{\mathrm{d}}x,$$
which implies,
\begin{eqnarray*}
F_A(v(t),\partial_tv(t),t)&\geq& \frac{A}{2} \displaystyle\int_{\R}(\partial_x v(t))^2  {\mathrm{d}}x+\frac{A-1}{2}\displaystyle\int_{\R}(\partial_tv(t))^2  {\mathrm{d}}x\\
&&+\frac{1}{8}\displaystyle\int_{\R}\frac{(v(t))^2}{(1+t)^2}  {\mathrm{d}}x+\frac{1}{4}\displaystyle\int_{\R}\frac{V(x)(v(t))^2}{1+t}  {\mathrm{d}}x.
\end{eqnarray*}
By choosing $A=\max (5/4,(1+R_0)/2\mu_0)$, this leads to
\begin{equation}\label{F17B}
F_A(v(t),\partial_tv(t),t)\geq \frac18 \displaystyle\int_{\R}\left(
\frac{(v(t))^2}{(t+1)^2}+\frac{V(x)(v(t))^2}{t+1}+
(\partial_x v(t))^2+(\partial_tv(t))^2 \right) {\mathrm{d}}x,\,\,\hbox{for all}\,\,t\geq s_0.
\end{equation}
Alternatively, using the standard Young's inequality, we conclude that
\begin{equation}\label{F18B}
F_A(v(t),\partial_tv(t),t)\leq \frac{2A+1}{4}\displaystyle\int_{\R}\left(
\frac{(v(t))^2}{(t+1)^2}+\frac{V(x)(v(t))^2}{t+1}+(\partial_x v(t))^2+(\partial_tv(t))^2 \right) {\mathrm{d}}x,\,\,\hbox{for all}\,\,t\geq s_0.
\end{equation}
From \eqref{F17B} and \eqref{F18B}, there exists a constant $C=C(R_0,\mu_0 )$ such that for all $t\ge s_0$,
\begin{equation} \label{F15B}
C^{-1} F_A(v(t),\partial_tv(t),t)\le
\displaystyle\int_{\R}\left(
\frac{(v(t))^2}{(t+1)^2}+\frac{V(x)(v(t))^2}{t+1}+
(\partial_x v(t))^2+(\partial_tv(t))^2
 \right) {\mathrm{d}}x\le C F_A(v(t),\partial_tv(t),t).
\end{equation}
Combining inequalities \eqref{F19B} and \eqref{F15B}, the proof of Lemma \ref{LEE} is complete.
\end{proof}

\subsection{Decay estimates of the $H^1\times L^2$ norm of $(v,\partial_tv)$ for $H^2\times H^1$ data}
To refine the decay estimate from Lemma  \ref{LEE},  it is necessary to introduce a different set of energy functionals depending on the parameter $\mu_0$. This leads us to consider two distinct cases.\\

{\bf First case:  $\mu_0 \in (0,1]$.}\\ 
Let $\mu\in(0,\mu_0)$. We introduce the following functionals:
\begin{eqnarray}
E_1(v(t),\partial_tv(t),t)&:=&\displaystyle\int_{\R}\big(v(t)\partial_tv(t)+ \frac{1-\mu}{2(t+1)}(v(t))^2+ \frac{V(x)}{2}(v(t))^2
 \big) {\mathrm{d}}x,\quad \label{In}\\
 E_2(v(t),\partial_tv(t),t)&:=&E_0(v(t),\partial_tv(t))+\frac{\mu  }{2(t+1)}E_1(v(t),\partial_tv(t),t),\quad\label{E3}
 \end{eqnarray}
for all $t\geq s_0$, where $E_0(v(t),\partial_tv(t))$ is given by \eqref{En}. By evaluating the time derivative of $E_i(v(t),\partial_tv(t),t)$, for $i=1,2$,  we are led to the following lemma.
\begin{lemma} \label{LE} 
Let  $\mu_0 \in (0,1]$ and $\mu\in (0,\mu_0)$. Assume
$(g(s_0),h(s_0))\in H^{2}(\mathbb{R})\times H^{1}(\mathbb{R})$ that are compactly supported 
on $(-R_0-s_0,R_0+s_0)$, with $R_0>0$, then there exists a constant $C_1=C_1(R_0,\mu,\mu_0)>0$ such that the strong solution $v$ of \eqref{1} satisfies, 
for all   $t\ge s_0$,
\begin{eqnarray} \label{A1}
	&&\frac{\|v(t)\|^2_{L^2}}{(1+t)^{2}}+\frac{\|\sqrt{V}v(t)\|^2_{L^2}}{1+t}+	\|\partial_tv(t)\|^2_{L^2}+\|\partial_x v(t)\|^2_{L^2}\nonumber\\
	&&\leq C_1
 \left(\frac{1+s_0}{1+t}\right)^\mu
\left( \frac{\|g(s_0)\|^2_{L^2}}{(1+s_0)^2}+\frac{\|\sqrt{V}g(s_0)\|^2_{L^2}}{1+s_0}+\|h(s_0)\|^2_{L^2}+\|\partial_x g(s_0)\|^2_{L^2}\right).
\end{eqnarray}
\end{lemma}
\begin{proof}	
By multiplying the equation in \eqref{1} by $v(t)$ and applying integration by parts, we obtain the following identity for all $t>s_0$,
\begin{eqnarray}\label{E1}
 \frac{d}{dt}E_1(v(t),\partial_tv(t),t)&=&\displaystyle\int_{\R}
(\partial_tv(t))^2  {\mathrm{d}}x
-\displaystyle\int_{\R}
(\partial_x v(t))^2  {\mathrm{d}}x\\
&&-\frac{1-\mu}{2(1+t)^2}\displaystyle\int_{\R}
(v(t))^2  {\mathrm{d}}x+\frac{1-\mu}{1+t}\displaystyle\int_{\R}
v(t)\partial_tv(t)  {\mathrm{d}}x.\nonumber
\end{eqnarray}
By exploiting \eqref{E00}, \eqref{E1} and  the definition of $E_2(v(t),\partial_tv(t),t)$ given in \eqref{E3}, we  infer that
\begin{eqnarray*}
 \frac{d}{dt}E_2(v(t),\partial_tv(t),t)&=&- \displaystyle\int_{\R}
V(x)(\partial_tv(t))^2  {\mathrm{d}}x+\frac{\mu}{2(1+t)}\displaystyle\int_{\R}
(\partial_tv(t))^2  {\mathrm{d}}x
-\frac{\mu}{2(1+t)}\displaystyle\int_{\R}
(\partial_x v(t))^2  {\mathrm{d}}x\\
&&-\frac{\mu(1-\mu)}{2(1+t)^3}\displaystyle\int_{\R}
(v(t))^2  {\mathrm{d}}x-\frac{\mu^2}{2(1+t)^2}\displaystyle\int_{\R}
v(t)\partial_tv(t)  {\mathrm{d}}x\\
&&-\frac{\mu}{4(1+t)^2 }\displaystyle\int_{\R}
V(x)(v(t))^2  {\mathrm{d}}x.
\end{eqnarray*}
Therefore,
\begin{eqnarray*}
 \frac{d}{dt}E_2(v(t),\partial_tv(t),t)+\frac{\mu}{1+t} 
 E_2(v(t),\partial_tv(t),t)&=&- \displaystyle\int_{\R}V(x)
(\partial_tv(t))^2  {\mathrm{d}}x+\frac{\mu}{1+t} \displaystyle\int_{\R}
(\partial_tv(t))^2  {\mathrm{d}}x\\
&&-\frac{\mu(1-\mu)(2-\mu)}{4(1+t)^3}\displaystyle\int_{\R}
(v(t))^2  {\mathrm{d}}x\\
&&-\frac{\mu (1-\mu)}{4(1+t)^2 }
\displaystyle\int_{\R}
V(x)(v(t))^2  {\mathrm{d}}x\\
&\leq&- \displaystyle\int_{\R}V(x)
(\partial_tv(t))^2  {\mathrm{d}}x+\frac{\mu}{1+t} \displaystyle\int_{\R}
(\partial_tv(t))^2  {\mathrm{d}}x.
\end{eqnarray*}
Using the fact that $\supp(v)\subseteq \{|x| \le R_0+t\}$, we have
 $V(x)\ge \frac{\mu_0}{\sqrt{1+(R_0+t)^2}}\ge \frac{\mu_0}{R_0+1+t}$, then there exists $t_0=t_0(R_0,\mu,\mu_0)=\frac{\mu R_0}{\mu_0-\mu}>0$ such that $V(x)\ge \mu/(1+t)$ for all $t\geq t_0$. Hence, for all $t\ge \tau_0:=\max\{s_0,t_0\}$, it follows that
\begin{equation}\label{E12}
 \frac{d}{dt}E_2(v(t),\partial_tv(t),t)+\frac{\mu}{1+t} 
 E_2(v(t),\partial_tv(t),t)\leq - \displaystyle\int_{\R}\left(V(x)-\frac{\mu}{1+t}\right)
(\partial_tv(t))^2  {\mathrm{d}}x\leq 0.
\end{equation}
Multiplying \eqref{E12} by  $(1+t)^{\mu}$ and integrating over the interval $[\tau_0, t]$, we deduce that, for all $t\ge \tau_0$,
 \begin{equation}\label{E14}
 (1+t)^{\mu} E_2(v(t),\partial_tv(t),t)\leq
 (1+\tau_0)^{\mu} E_2(v(\tau_0),\partial_tv(\tau_0),\tau_0).
\end{equation}
In addition, applying Young's inequality,  we obtain
$$\frac{\mu}{2}\displaystyle\int_{\R}\frac{v(t)\partial_t v(t)}{1+t}{\mathrm{d}}x\geq -\frac{\mu}{8}\int_{\R}\frac{(v(t))^2}{(1+t)^2}{\mathrm{d}}x-\frac{\mu}{2}\int_{\R}(\partial_t v(t))^2{\mathrm{d}}x.$$
Thus, for all $t\geq s_0$, we have
\begin{eqnarray*}
E_2(v(t),\partial_tv(t),t)&\geq& \frac{1}{2}\displaystyle\int_{\R}(\partial_x v(t))^2  {\mathrm{d}}x+\frac{1-\mu}{2}\displaystyle\int_{\R}(\partial_tv(t))^2  {\mathrm{d}}x\\
&&+\frac{\mu-2\mu^2}{8}\displaystyle\int_{\R}\frac{(v(t))^2}{(1+t)^2}  {\mathrm{d}}x+\frac{\mu}{8}\displaystyle\int_{\R}\frac{V(x)(v(t))^2}{1+t}  {\mathrm{d}}x+\frac{\mu}{8}\displaystyle\int_{\R}\frac{V(x)(v(t))^2}{1+t}  {\mathrm{d}}x.
\end{eqnarray*}
Plugging the inequality $V(x)\ge {\mu}/{(1+t)}$ valid for all $t\geq t_0$ into the above estimate, we obtain
$$E_2(v(t),\partial_tv(t),t)\geq \frac{1}{2}\displaystyle\int_{\R}(\partial_x v(t))^2  {\mathrm{d}}x+\frac{1-\mu}{2}\displaystyle\int_{\R}(\partial_tv(t))^2  {\mathrm{d}}x+\frac{\mu(1-\mu)}{8}\displaystyle\int_{\R}\frac{(v(t))^2}{(1+t)^2}  {\mathrm{d}}x+\frac{\mu}{8}\displaystyle\int_{\R}\frac{V(x)(v(t))^2}{1+t}  {\mathrm{d}}x,$$
for all $t\geq \tau_0$. Consequently, we have  
\begin{equation}\label{E17}
E_2(v(t),\partial_tv(t),t)\geq C\displaystyle\int_{\R}\left(
\frac{(v(t))^2}{(t+1)^2}+\frac{V(x)(v(t))^2}{t+1}+
(\partial_t v(t))^2+
(\partial_xv(t))^2 \right) {\mathrm{d}}x,\,\,\hbox{for all}\,t\geq \tau_0. 
\end{equation}
On the other hand, using Young's inequality, we conclude that a constant $C=C(\mu )$ exists such that
\begin{equation}\label{E18}
E_2(v(t),\partial_tv(t),t)\leq C\displaystyle\int_{\R}\left(
\frac{(v(t))^2}{(t+1)^2}+\frac{V(x)(v(t))^2}{t+1}+
(\partial_t v(t))^2+(\partial_xv(t))^2 \right) {\mathrm{d}}x,\,\,\hbox{for all}\,t\geq s_0.
\end{equation}
Combining \eqref{E17} and \eqref{E18}, it follows that there exists $C=C(\mu )$ such that  
\begin{equation} \label{E15}
C^{-1} E_2(v(t),\partial_tv(t),t)\le
\displaystyle\int_{\R}\left(
\frac{(v(t))^2}{(t+1)^2}+\frac{V(x)(v(t))^2}{t+1}+(\partial_tv(t))^2+
(\partial_x v(t))^2
 \right) {\mathrm{d}}x\le C E_2(v(t),\partial_tv(t),t),
\end{equation}
for all $t\ge \tau_0$. In conclusion, by using \eqref{E14}, and \eqref{E15},  we conclude that
\begin{eqnarray} \label{E16}
&&\displaystyle\int_{\R}\left(\frac{(v(t))^2}{(t+1)^2}+\frac{V(x)(v(t))^2}{t+1}+(\partial_t
v(t))^2+
(\partial_xv(t))^2
 \right) {\mathrm{d}}x\notag\\
 &&\le C^2
 \left(\frac{1+s_0}{1+t}\right)^{\mu}\left( \frac{\|g(s_0)\|^2_{L^2}}{(1+s_0)^2}+\frac{\|\sqrt{V}g(s_0)\|^2_{L^2}}{1+s_0}
+	\|h(s_0)\|^2_{L^2}+\|\partial_x g(s_0)\|^2_{L^2}\right),
\end{eqnarray}
for all $t\geq s_0$ when $s_0\ge t_0$, while, if $ s_0\le t_0$, then
\begin{eqnarray} \label{E166}
&&\displaystyle\int_{\R}\left(\frac{(v(t))^2}{(t+1)^2}+\frac{V(x)(v(t))^2}{t+1}+(\partial_t
v(t))^2+
(\partial_xv(t))^2
 \right) {\mathrm{d}}x\notag\\
 &&\le C^2
 \left(\frac{1+\tau_0}{1+t}\right)^{\mu}\displaystyle\int_{\R}\left(
 \frac{(v(\tau_0))^2}{(\tau_0+1)^2}
+\frac{V(x)(v(\tau_0))^2}{\tau_0+1} + (\partial_t
v(\tau_0))^2+
(\partial_xv(\tau_0))^2 \right) {\mathrm{d}}x\notag\\
 &&\le C^2\tilde{C_1} 
 \left(\frac{1+s_0}{1+t}\right)^{\mu}\left( 	\frac{\|g(s_0)\|^2_{L^2}}{(1+s_0)^2}+\frac{\|\sqrt{V}g(s_0)\|^2_{L^2}}{1+s_0}
+	\|h(s_0)\|^2_{L^2}+\|\partial_x g(s_0)\|^2_{L^2}\right),
\end{eqnarray}
for all $t\geq t_0$, where we have used Lemma \ref{LEE} with $\tilde{C_1}=C_0(1+t_0)^\mu=C_0\left(1+\frac{\mu R_0}{\mu_0-\mu}\right)^\mu$. Furthermore, when $s_0\leq t_0$, by exploiting Lemma \ref{LEE}, we deduce that for all $ t \in[s_0, t_0]$,
\begin{eqnarray}\label{E16C}
&&\frac{\|v(t)\|^2_{L^2}}{(1+t)^{2}}+\frac{\|\sqrt{V}v(t)\|^2_{L^2}}{1+t}
+	\|\partial_tv(t)\|^2_{L^2}+\|\partial_x v(t)\|^2_{L^2}\notag\\
&&\leq \tilde{C_1} \left(\frac{1+s_0}{1+t}\right)^{\mu}
\left( 	\frac{\|g(s_0)\|^2_{L^2}}{(1+s_0)^2}+\frac{\|\sqrt{V}g(s_0)\|^2_{L^2}}{1+s_0}
+	\|h(s_0)\|^2_{L^2}+\|\partial_x g(s_0)\|^2_{L^2}\right).
 \end{eqnarray}
Combining \eqref{E16}, \eqref{E166}, and \eqref{E16C}, the proof of Lemma \ref{LE} is complete.
\end{proof}	

{\bf Second case:  $\mu_0 \in (1,\infty)$.}\\ 
We   introduce  the following functionals for all $t\geq s_0$:
\begin{eqnarray*}
E_3(v(t),\partial_tv(t))&:=&\displaystyle\int_{\R}\big(v(t)\partial_tv(t)+ \frac{V(x)}{2}(v(t))^2
 \big){\mathrm{d}}x,\\
 E_4(v(t),\partial_tv(t),t)&:=&E_0(v(t),\partial_tv(t))+\frac{1}{2(t+1)}E_3(v(t),\partial_tv(t)),
\end{eqnarray*}
where $E_0(v(t),\partial_tv(t))$ is defined in \eqref{En}.
\begin{lemma} \label{LEA} 
Let $\mu_0>1$. Assume
$(g(s_0),h(s_0))\in H^{2}(\mathbb{R})\times H^{1}(\mathbb{R})$ that are compactly supported 
on $(-R_0-s_0,R_0+s_0)$, with $R_0>0$, then there exists a constant $C_2=C_2(R_0,\mu_0)>0$ such that the strong solution $v$ of \eqref{1} satisfies, 
for all   $t\ge s_0$,
\begin{eqnarray} \label{A1A}
	&&\frac{\|v(t)\|^2_{L^2}}{(1+t)^{2}}+\frac{\|\sqrt{V}v(t)\|^2_{L^2}}{1+t}+	\|\partial_tv(t)\|^2_{L^2}+\|\partial_x v(t)\|^2_{L^2}\nonumber\\
	&&\leq C
 \left(\frac{1+s_0}{1+t}\right)
\left( \frac{\|g(s_0)\|^2_{L^2}}{(1+s_0)^2}+\frac{\|\sqrt{V}g(s_0)\|^2_{L^2}}{1+s_0}+\|h(s_0)\|^2_{L^2}+\|\partial_x g(s_0)\|^2_{L^2}\right).
\end{eqnarray}
\end{lemma}
\begin{proof}	
By multiplying the differential equation in \eqref{1} by $v(t)$, and applying integration by parts, we obtain the following for all $t>s_0$,
\begin{equation}\label{E3A}
 \frac{d}{dt}E_3(v(t),\partial_tv(t))=\int_{\R}(\partial_tv(t))^2  {\mathrm{d}}x-\int_{\R}(\partial_x v(t))^2  {\mathrm{d}}x.
\end{equation}
Hence,  by exploiting  \eqref{E00}, \eqref{E3A} and  the definition of $E_4(v(t),\partial_tv(t),t)$, we  infer that
\begin{eqnarray*}
 \frac{d}{dt}E_4(v(t),\partial_tv(t),t)&=&- \displaystyle\int_{\R}
V(x)(\partial_tv(t))^2  {\mathrm{d}}x+\frac{1}{2(1+t)}\displaystyle\int_{\R}
(\partial_tv(t))^2  {\mathrm{d}}x
-\frac{1}{2(1+t)}\displaystyle\int_{\R}
(\partial_x v(t))^2  {\mathrm{d}}x\\
&&-\frac{1}{2(1+t)^2}\displaystyle\int_{\R}
v(t)\partial_tv(t)  {\mathrm{d}}x-\frac{1}{4(1+t)^2 }\displaystyle\int_{\R}
V(x)(v(t))^2  {\mathrm{d}}x.
\end{eqnarray*}
Consequently,
$$
 \frac{d}{dt}E_4(v(t),\partial_tv(t),t)+\frac{1}{1+t} 
 E_4(v(t),\partial_tv(t),t)=- \displaystyle\int_{\R}\left(V(x)-\frac{1}{1+t} \right)
(\partial_tv(t))^2  {\mathrm{d}}x.
$$
Noting that $\supp(v)\subseteq \{|x| \le R_0+t\}$, it follows that
 $V(x)\ge \frac{\mu_0}{\sqrt{1+(R_0+t)^2}}\ge \frac{\mu_0}{R_0+1+t}$. Hence, there exists $t_1=t_1(R_0,\mu_0)=\frac{R_0}{\mu_0-1}>0$ such that $V(x)\ge 1/(1+t)$ for all $t\geq t_1$. Therefore, we obtain
 \begin{equation}\label{E12A}
 \frac{d}{dt}E_4(v(t),\partial_tv(t),t)+\frac{1}{1+t} 
 E_4(v(t),\partial_tv(t),t)\leq 0,\quad\hbox{for all}\,\,t\ge \tau_1:=\max\{s_0,t_1\}.
\end{equation}
Multiplying \eqref{E12A} by  $(1+t)$ and integrating over $[\tau_1, t]$, we deduce that, for all $t\ge \tau_1$,
 \begin{equation}\label{E14A}
 (1+t) E_4(v(t),\partial_tv(t),t)\leq
 (1+\tau_1) E_4(v(\tau_1),\partial_tv(\tau_1),\tau_1).
\end{equation}
In addition, applying the $\varepsilon$-Young's inequality, we have
$$\frac{1}{2}\displaystyle\int_{\R}\frac{v(t)\,\partial_t v(t)}{1+t}{\mathrm{d}}x\geq -\frac{\varepsilon}{4}\int_{\R}\frac{(v(t))^2}{(1+t)^2}{\mathrm{d}}x-\frac{1}{4\varepsilon}\int_{\R}(\partial_t v(t))^2{\mathrm{d}}x.$$
Thus, using  the fact that $V(x)\ge {1}/{(1+t)}$ for all $t\geq t_1$, we conclude that
\begin{eqnarray*}
E_4(v(t),\partial_tv(t),t)&\geq& \frac{1}{2}\displaystyle\int_{\R}(\partial_x v(t))^2  {\mathrm{d}}x+\frac{1}{2}\left(1-\frac{1}{2\varepsilon}\right)\displaystyle\int_{\R}(\partial_tv(t))^2  {\mathrm{d}}x\\
&&+\frac{1}{16}\displaystyle\int_{\R}\frac{V(x)(v(t))^2}{1+t}  {\mathrm{d}}x+\frac{3-4\varepsilon}{16}\displaystyle\int_{\R}\frac{(v(t))^2}{(1+t)^2}  {\mathrm{d}}x.
\end{eqnarray*}
Choosing $\varepsilon=5/8$, we derive
\begin{equation}\label{E17A}
E_4(v(t),\partial_tv(t),t)\geq \frac{1}{32}\displaystyle\int_{\R}\left(
\frac{(v(t))^2}{(t+1)^2}+\frac{V(x)(v(t))^2}{t+1}+
(\partial_tv(t))^2+
(\partial_xv(t))^2 \right) {\mathrm{d}}x,\qquad\hbox{for all}\,\,t\geq \tau_1.
\end{equation}
On the other hand, employing  Young's inequality, we conclude that
\begin{equation}\label{E18A}
E_4(v(t),\partial_tv(t),t)\leq \frac{3}{4}\displaystyle\int_{\R}\left(
\frac{(v(t))^2}{(t+1)^2}+\frac{V(x)(v(t))^2}{t+1}+
(\partial_tv(t))^2+
(\partial_xv(t))^2 \right) {\mathrm{d}}x,\qquad\hbox{for all}\,\,t\geq s_0.
\end{equation}
It follows from \eqref{E17A} and \eqref{E18A} that
\begin{equation} \label{E15A}
{C}^{-1} E_4(v(t),\partial_tv(t),t)\le
\displaystyle\int_{\R}\left(
\frac{(v(t))^2}{(t+1)^2}+\frac{V(x)(v(t))^2}{t+1}+
(\partial_x
v(t))^2+(\partial_tv(t))^2
 \right) {\mathrm{d}}x\le {C} E_4(v(t),\partial_tv(t),t),
\end{equation}
for some $C>0$, for all $t\ge \tau_1$. To conclude, using \eqref{E14A}, and \eqref{E15A},  we deduce that
\begin{eqnarray} \label{E16A}
&&\displaystyle\int_{\R}\left(\frac{(v(t))^2}{(t+1)^2}+\frac{V(x)(v(t))^2}{t+1}+(\partial_t
v(t))^2+
(\partial_xv(t))^2
 \right) {\mathrm{d}}x\notag\\
 &&\le C^2
 \left(\frac{1+s_0}{1+t}\right)\left( \frac{\|g(s_0)\|^2_{L^2}}{(1+s_0)^2}+\frac{\|\sqrt{V}g(s_0)\|^2_{L^2}}{1+s_0}
+	\|h(s_0)\|^2_{L^2}+\|\partial_x g(s_0)\|^2_{L^2}\right),
\end{eqnarray}
for all $t\geq s_0$ when $s_0\ge t_1$, while, if $ s_0\le t_1$, then for all $t\geq t_1$,
\begin{eqnarray} \label{E166A}
&&\displaystyle\int_{\R}\left(\frac{(v(t))^2}{(t+1)^2}+\frac{V(x)(v(t))^2}{t+1}+(\partial_t
v(t))^2+
(\partial_xv(t))^2
 \right) {\mathrm{d}}x\notag\\
 &&\le C^2
 \left(\frac{1+\tau_1}{1+t}\right)\displaystyle\int_{\R}\left(
 \frac{(v(\tau_1))^2}{(\tau_1+1)^2}
+\frac{V(x)(v(\tau_1))^2}{\tau_1+1} + (\partial_t
v(\tau_1))^2+
(\partial_xv(\tau_1))^2 \right) {\mathrm{d}}x\notag\\
 &&\le C^2\tilde{C_1} 
 \left(\frac{1+s_0}{1+t}\right)\left( \frac{\|g(s_0)\|^2_{L^2}}{(1+s_0)^2}+\frac{\|\sqrt{V}g(s_0)\|^2_{L^2}}{1+s_0}
+	\|h(s_0)\|^2_{L^2}+\|\partial_x g(s_0)\|^2_{L^2}\right),
\end{eqnarray}
where we have used Lemma \ref{LEE}, with $\tilde{C_2}=C_0(1+t_1)^\mu=C_0\left(1+\frac{R_0}{\mu_0-1}\right)$. Furthermore, when $s_0\leq t_1$, by exploiting Lemma \ref{LEE}, we deduce that for all $ t \in[s_0, t_1]$,
\begin{eqnarray}\label{E16D}
&&\frac{\|v(t)\|^2_{L^2}}{(1+t)^{2}}+\frac{\|\sqrt{V}v(t)\|^2_{L^2}}{1+t}
+	\|\partial_tv(t)\|^2_{L^2}+\|\partial_x v(t)\|^2_{L^2}\notag\\
&&\leq \tilde{C_2} \left(\frac{1+s_0}{1+t}\right)
\left( 	\frac{\|g(s_0)\|^2_{L^2}}{(1+s_0)^2}+\frac{\|\sqrt{V}g(s_0)\|^2_{L^2}}{1+s_0}
+	\|h(s_0)\|^2_{L^2}+\|\partial_x g(s_0)\|^2_{L^2}\right).
 \end{eqnarray}
The proof of Lemma \ref{LEA} is completed by combining \eqref{E16A}, \eqref{E166A}, and \eqref{E16D} .
\end{proof}	

\subsection{Decay estimates of the $H^1\times L^2$ norm of $(v,\partial_tv)$ for $H^1\times L^2$ data}\label{subs2.3}
Based on the  Lemmas proved in the previous subsections, 
we now proceed to prove the  main result of this section, namely  Proposition
\ref{prop2.1}. To this end, we begin by introducing the operator $R(t,s_0)$ that maps the initial data $(v(s_0),\partial_tv(s_0))\in
H^{2}(\mathbb{R})\times H^{1}(\mathbb{R})$, given at time $s_0\geq0$, to the solution $v(t)\in
H^{2}(\mathbb{R})$ at time $t \geq s_0$. That is, the solution $v$ of \eqref{1}
is expressed as $v(t)=R(t,s_0)(g(s_0),h(s_0))$.  We also define the operator $S(t,s_0)k:=R(t,s_0)(0,k)$ for any function $k\in H^{1}(\mathbb{R})$. Moreover, since $R(t,s_0)$ depends only on the difference $t-s_0$, we write $R(t-s_0)$ instead of $R(t,s_0)$, particularly $R(t)\equiv R(t,0)$. Finally, if $(g(s_0),h(s_0))\in H^{1}(\mathbb{R})\times L^{2}(\mathbb{R})$, the function $t\mapsto R(t,s_0)(g(s_0),h(s_0))$, defined in the sense of \eqref{Extend.Operator1} below, is called a {\it generalized
solution} of the initial value problem \eqref{1}, also referred to as a {\it mild} solution.

Let  $\mu \in (0,\mu_0)$ and let 
\begin{equation}\label{alpha}
\alpha:=\alpha(\mu,\mu_0)=
\left\{
\begin{array}{ll}
 \mu\  &\textrm{  if}\ \  \ \mu_0\in(0,1],\\\\
 1 &\textrm{  if}\ \  \ \mu_0 >1.
\end{array}
\right.
\end{equation} 
\begin{proposition}\label{prop2.1}
 Assume $(g(s_0),h(s_0))\in H^{1}(\mathbb{R})\times L^{2}(\mathbb{R})$ that are compactly supported 
on $(-R_0-s_0,R_0+s_0)$, with $R_0>0$, then there exist a constant $C=C(R_0,\alpha,\alpha_0)>0$ and a unique {\bf mild} solution $v\in  \mathcal{C}\left([s_0,\infty),H^1(\mathbb{R})\right)\cap \mathcal{C}^{1}\left([s_0,\infty),L^2(\mathbb{R})\right)$ of \eqref{1} satisfying
\begin{eqnarray}\label{0A1}
&&\frac{\|v(t)\|^2_{L^2}}{(1+t)^{2}}+\frac{\|\sqrt{V}v(t)\|^2_{L^2}}{1+t}
+	\|\partial_tv(t)\|^2_{L^2}+\|\partial_x v(t)\|^2_{L^2}\notag\\
&&\leq C \left(\frac{1+s_0}{1+t}\right)^{\alpha}
\left( \frac{\|v(s_0)\|^2_{L^2}}{(1+s_0)^2}+\frac{\|\sqrt{V}v(s_0)\|^2_{L^2}}{1+s_0}
+	\|\partial_tv(s_0)\|^2_{L^2}+\|\partial_x v(s_0)\|^2_{L^2}\right),
 \end{eqnarray}
for  all   $t\ge s_0$, where $\alpha$ is defined in \eqref{alpha}.  In addition,  if $(g(s_0),h(s_0))\in H^{2}(\mathbb{R})\times H^{1}(\mathbb{R})$, then $v$ is a  {\bf strong} solution of \eqref{1} and satisfies
\begin{eqnarray}\label{5mai1}
&&\frac{\|\sqrt{V}\partial_tv(t)\|^2_{L^2}}{1+t}+\|\partial^2_t v(t)\|^2_{L^2}+\|\partial^2_{xt}v(t)\|^2_{L^2}+\|\partial_x^2  v(t)\|^2_{L^2}\notag\\
&&\leq C
 \left(\frac{1+s_0}{1+t}\right)^{\alpha}
\left( \frac{\|\partial_t v(s_0)\|^2_{L^2}}{
	(1+s_0)^2}+\frac{\|\sqrt{V}\partial_tv(s_0)\|^2_{L^2}}{1+s_0}
+	\|\partial^2_tv(s_0)\|^2_{L^2}+\|\partial^2_{xt}v(s_0)\|^2_{L^2}\right),
 \end{eqnarray}
for  all  
  $t\ge s_0$.
\end{proposition}
\begin{proof}
{\bf Existence and Uniqueness.} Let $T_0>s_0$ be arbitrary, and suppose that $(g(s_0),h(s_0)\in H^{1}(\mathbb{R})\times L^{2}(\mathbb{R})$. By a density argument, there exists a sequence $\{(g_{s_0}^{(j)},h_{s_0}^{(j)})\}_{j=1}^\infty\subseteq
H^2(\mathbb{R})\times H^1(\mathbb{R})$, such that 
$$\supp(g_{s_0}^{(j)}),\  \supp(h_{s_0}^{(j)})\subseteq (-1-R_0-s_0,1+R_0+s_0),\qquad\hbox{for all}\,\,j\in\mathbb{N},$$ 
and
$$\lim_{j\rightarrow\infty}(g_{s_0}^{(j)},h_{s_0}^{(j)})=(g(s_0),h(s_0))\;\;\mbox{in}\;\; H^1(\mathbb{R})\times L^2(\mathbb{R}).$$
Using Theorem \ref{existencetheorem}, let $v^{(j)}$ denote the strong solution of the linear homogeneous equation \eqref{1} with initial data $(g_{s_0}^{(j)},h_{s_0}^{(j)})$. Then, the difference $v^{(j)}-v^{(k)}$ is a strong solution of the Cauchy problem 
$$
\begin{cases}
\displaystyle \partial_{t}^2v-\partial_x^2 v +V(x) \partial_{t}v =0, &x\in \mathbb{R},\,t>s_0,\\\\
v(x,s_0)= g_{s_0}^{(j)}(x)-g_{s_0}^{(k)}(x),\, \partial_{t}v(x,s_0)=  h_{s_0}^{(j)}(x)-h_{s_0}^{(k)}(x),&x\in \mathbb{R}.\\
\end{cases}
$$
By applying Lemmas \ref{LE},  and \ref{LEA},  to $v^{(j)}-v^{(k)}$, we obtain, in particular, the following estimate
\begin{eqnarray*}
&&\|v^{(j)}-v^{(k)}\|^2_{L^2}+\|\partial_t(v^{(j)}-v^{(k)})\|^2_{L^2}+\|\partial_x (v^{(j)}-v^{(k)})\|^2_{L^2}\\
&&\leq C (1+T_0)^{2}\left( \|g_{s_0}^{(j)}-g_{s_0}^{(k)}\|^2_{L^2}+\|h_{s_0}^{(j)}-h_{s_0}^{(k)}\|^2_{L^2}+\|\partial_x (g_{s_0}^{(j)}-g_{s_0}^{(k)})\|^2_{L^2}\right),
\end{eqnarray*}
for  all  $t\in [s_0,T_0]$.  This shows that the sequence $\left\{v^{(j)}\right\}_{j=1}^\infty$ is Cauchy in the complete space $\mathcal{C}([s_0,T_0];H^1(\mathbb{R}))\cap \mathcal{C}^1([s_0,T_0];L^2(\mathbb{R}))$. Therefore, we can define its limit 
\begin{equation}\label{eq2.9}
\lim_{j\rightarrow\infty}v^{(j)}=v\in \mathcal{C}([s_0,\infty);H^1(\mathbb{R}))\cap \mathcal{C}^1([s_0,\infty);L^2(\mathbb{R})),
\end{equation}
since $T_0>s_0$ is arbitrary. As $v^{(j)}$ satisfies $v^{(j)}(t,x)=R(t,s_0)(g_{s_0}^{(j)},h_{s_0}^{(j)})$, it follows that
$v(t)=\lim\limits_{j\rightarrow\infty} R(t,s_0)(g_{s_0}^{(j)},h_{s_0}^{(j)})$. This shows that the operator $R(t,s_0)$ can be uniquely extended to a new operator
	\begin{align}\label{Extend.Operator1}
	\widetilde{R}(t,s_0)&:\,H^1(\mathbb{R})\times L^2(\mathbb{R})\,\longrightarrow\, X_{s_0}\\
	&\qquad(g(s_0),h(s_0))\,\,\longmapsto v(t)\notag
	\end{align}
	also denoted by $R(t,s_0)$, where
	$X_{s_0}:=\mathcal{C}\left([s_0,\infty),H^1(\mathbb{R})\right)\cap
	\mathcal{C}^1\left([s_0,\infty),L^2(\mathbb{R})\right)$. Therefore, $v(t)=R(t,s_0)(g(s_0),h(s_0))$ is the unique mild solution of the problem \eqref{1}.\\ 
{\bf Energy estimate \eqref{0A1}.} By Lemmas \ref{LE}, and \ref{LEA}, each strong solution $v^{(j)}$ constructed above satisfies the estimates \eqref{A1}, \eqref{A1A}, with $R_0$ replaced by $R_0+1$. By letting $j\rightarrow\infty$ and using the convergence in (\ref{eq2.9}), the same estimates carry over to the mild solution $v$. Hence, the energy estimate \eqref{0A1} holds.\\
{\bf Energy estimate \eqref{5mai1}.} A straightforward computation shows that $w=\partial_t v$ satisfies
\begin{equation}\label{ut}
\begin{cases}
\partial_{t}^2w-\partial_x^2 w +V(x) \partial_{t}w=0, &x\in \mathbb{R},\,t>s_0,\\
w(x,s_0)=  h(x,s_0),\,\partial_{t}w(x,s_0)=  \partial_x^2 g(x,s_0)-V(x)h(x,s_0),&x\in \mathbb{R}.\\
\end{cases}
\end{equation}
Since $(g(s_0),h(s_0))\in H^{2}(\mathbb{R})\times H^{1}(\mathbb{R})$, it follows that $(w(s_0),\partial_tw(s_0))\in H^{1}(\mathbb{R})\times L^{2}(\mathbb{R})$. Then, applying estimate \eqref{0A1} to the mild solution $w$, we obtain for  all $t\ge s_0$,
\begin{eqnarray*}
&&\frac{\|w(t)\|^2_{L^2}}{(1+t)^{2}}+\frac{\|\sqrt{V}w(t)\|^2_{L^2}}{1+t}
+	\|\partial_tw(t)\|^2_{L^2}+\|\partial_x w(t)\|^2_{L^2}\\
&&\leq C \left(\frac{1+s_0}{1+t}\right)^{\alpha}
\left( 
	\frac{\|w(s_0)\|^2_{L^2}}{(1+s_0)^2}+\frac{\|\sqrt{V}w(s_0)\|^2_{L^2}}{1+s_0}
+	\|\partial_tw(s_0)\|^2_{L^2}+\|\partial_x w(s_0)\|^2_{L^2}\right).
 \end{eqnarray*}
This, combined with the relation $\partial_x^2 v=\partial_{t}^2v +V(x) \partial_{t}v$, leads to estimate \eqref{5mai1}. This completes the proof of Proposition \ref{prop2.1}.
\end{proof}
\begin{remark}\label{rmk2.2}
It is worth noting that Proposition \ref{prop2.1} can be extended to higher dimensions ($n\geq 2$). However, this extension does not allow us to generalize the global existence result of Theorem \ref{globalexistence1}, due to the requirement of the Sobolev embedding $H^1(\R^n)\hookrightarrow L^\infty(\R^n)$, which holds only for $n=1$. For this reason, our analysis in this paper is restricted to the one-dimensional case, specifically problem to \eqref{NLW}.
\end{remark}


\section{Proof of Theorems  \ref{globalexistence1}}\label{sec3}

This section  is devoted to proving Theorem \ref{globalexistence1}. To manage the nonlinear term, we begin by recalling some useful elementary inequalities and lemmas. \begin{lemma}\label{basic}
Let $p\geq1, a,b\in  \R$. Then 
\begin{eqnarray}
||a|^p-|b|^p|&\leq &
C (|a|^{p-1}+|b|^{p-1}) |a-b|, \label{ab}\\
||a|^{p-1}a-|b|^{p-1}b|&\leq &
C (|a|^{p-1}+|b|^{p-1})|a-b|. \label{abb}
\end{eqnarray}
\end{lemma}
Then, we are in the position to  
state and prove  the  following nonlinear estimates.
\begin{lemma}\label{handle}
Let $p\geq2$ and let $F,G:\mathbb{R}\rightarrow\mathbb{R}$ be defined by
$$F(u(x,t))=|u(x,t)|^p\quad\hbox{and}\quad G(u(x,t))=|u(x,t)|^{p-1}u(x,t),$$ 
for all functions $u:\mathbb{R}\times\mathbb{R}_+\rightarrow\mathbb{R}$. Then, for all  $u,v\in \mathcal{C}^1\left([0,T],H^1(\mathbb{R})\right)$, we have
\begin{equation}\label{21A}
|F(u(x,t))-F(v(x,t))|\leq
C(|u(x,t)|^{p-1}+|v(x,t)|^{p-1}) |u(x,t)-v(x,t)|,
\end{equation}
\begin{equation}\label{21AG}
|G(u(x,t))-G(v(x,t))|\leq
C(|u(x,t)|^{p-1}+|v(x,t)|^{p-1}) |u(x,t)-v(x,t)|,
\end{equation}
\begin{eqnarray}\label{21B}
&{}&\big|\partial_x \big[F(u(x,t))\big]-\partial_x \big[F(v(x,t))\big]\big|\\
&{}&\leq C \ |u(x,t)|^{p-1}|\partial_x (u(x,t)-v(x,t))|+C |\partial_x v(x,t) |
(|u(x,t)|^{p-2}+|v(x,t)|^{p-2}) |u(x,t)-v(x,t)|,\nonumber
\end{eqnarray}
\begin{eqnarray}\label{21BG}
&{}&\big|\partial_x \big[G(u(x,t))\big]-\partial_x \big[G(v(x,t))\big]\big|\\
&{}&\leq C \ |u(x,t)|^{p-1}|\partial_x (u(x,t)-v(x,t))|+C |\partial_x v(x,t) |
(|u(x,t)|^{p-2}+|v(x,t)|^{p-2}) |u(x,t)-v(x,t)|,\nonumber
\end{eqnarray}
\begin{eqnarray}\label{21C}
&{}&\big|\partial_t \big[F(u(x,t))\big]-\partial_t  \big[F(v(x,t))\big]\big|\\
&{}&\leq C \ |u(x,t)|^{p-1}|\partial_t  (u(x,t)-v(x,t))|+C |\partial_t  v(x,t) |
(|u(x,t)|^{p-2}+|v(x,t)|^{p-2}) |u(x,t)-v(x,t)|,\nonumber
\end{eqnarray}
and 
\begin{eqnarray}\label{21BC}
&{}&\big|\partial_t  \big[G(u(x,t))\big]-\partial_t  \big[G(v(x,t))\big]\big|\\
&{}&\leq C \ |u(x,t)|^{p-1}|\partial_t  (u(x,t)-v(x,t))|+C |\partial_t  v(x,t) |
(|u(x,t)|^{p-2}+|v(x,t)|^{p-2}) |u(x,t)-v(x,t)|,\nonumber
\end{eqnarray}
 for all $t\in[0,T]$ and almost every \(x\in\mathbb{R}\).
\end{lemma}

\begin{proof}
It is worth noting that inequalities \eqref{21A} and \eqref{21AG} follow directly from  \eqref{ab} and \eqref{abb}, respectively. To estimate \eqref{21B}, observe that for any $w\in \mathcal{C}^1(\mathbb{R})$,
$$\partial_xF(w) = p |w|^{p-1}\text{sgn}(w) \partial_x w= p |w|^{p-2}w \partial_x w.$$
A straightforward calculation shows that, whenever  $u(\cdotp,t), v(\cdotp,t) \in \mathcal{C}^1(\mathbb{R})$, for all $x\in\mathbb{R}$ and $t\in[0,T]$, 
\begin{eqnarray}\label{20A}
\partial_xF(u(x,t))-\partial_xF(v(x,t))&=&p\Big(|u(x,t)|^{p-2}u(x,t)\partial_xu(x,t)-|v(x,t)|^{p-2}v(x,t)\partial_xv(x,t)\Big)\nonumber\\
&=&p|u(x,t)|^{p-2}u(x,t)(\partial_xu(x,t)-\partial_xv(x,t))\nonumber\\
&{}&\,+\,p\,
\partial_xv(x,t) (|u(x,t)|^{p-2}u(x,t)-|v(x,t)|^{p-2}v(x,t)).
\end{eqnarray}
 In fact, when $u,v\in \mathcal{C}^1\left([0,T],H^1(\mathbb{R})\right)$, by approximation, let $\{u_n(t)\},\{v_n(t)\}  \subset \mathcal{C}^\infty_c(\mathbb{R})$ denote the sequences such that, for all $t\in[0,T]$,
$$u_n(t) \to u(t) \qquad\hbox{and}\qquad v_n(t) \to v(t) \quad \text{in } H^1(\mathbb{R}).$$
Since $F\in C^1(\mathbb{R})$ has a bounded derivative and
\[
F'(u_n) = p |u_n|^{p-1}\text{sgn}(u_n)=p |u_n|^{p-2}u_n,\qquad F'(v_n) = p |v_n|^{p-1}\text{sgn}(v_n)=p |v_n|^{p-2}v_n, \quad\hbox{for all}\,\,n,
\]
we may apply the \textbf{Sobolev Chain Rule:} If $w \in H^1(\mathbb{R})$, $F \in \mathcal{C}^1(\mathbb{R})$ has a bounded derivative, then $F\circ w\in H^1(\mathbb{R})$ and
\[
\partial_xF(w) = F'(w) \partial_x w = p |w|^{p-2}w \partial_x w,
\]
in the distributional sense. Then, passing to the limit:
\[
\partial_{x}(F(u)) = \lim_{n \to \infty} \partial_{x}(F(u_n)) = \lim_{n \to \infty} F'(u_n) \, \partial_{x} u_n,\qquad \partial_{x}(F(v)) = \lim_{n \to \infty} \partial_{x}(F(v_n)) = \lim_{n \to \infty} F'(v_n) \, \partial_{x} v_n,
\]
and using dominated convergence or strong convergence in $L^2$, we get:
\[
\partial_{x}(F(u)) = p |u|^{p-1} \, \partial_{x} u\quad\hbox{and}\quad \partial_{x}(F(v)) = p |v|^{p-1} \, \partial_{x} v.
\]
Consequently, we derive from \eqref{20A} that
\begin{eqnarray}\label{21d}
\big|\partial_x \big[F(u)\big]-\partial_x \big[F(v)\big]\big|
&\leq& C \ |u|^{p-1}|\partial_x u-\partial_x v|+
C |\partial_x v |\ 
\big|  |u|^{p-2}u-|v|^{p-2}v\big|,
\end{eqnarray}
for all $t\in[0,T]$ and almost everywhere in \( \mathbb{R} \). Taking into account inequality 
\eqref{abb}, the condition $p\geq2$, and relation \eqref{21d}, we obtain estimate \eqref{21B}. Following the same reasoning, we derive \eqref{21BG} by using $G'(w) = p |w|^{p-1}$, for all $w\in \mathcal{C}^1(\mathbb{R})$. Likewise, estimates \eqref{21C}-\eqref{21BC} follow analogously by using the fact that $u\in \mathcal{C}^1\left([0,T],H^1(\mathbb{R})\right)$ implies that for almost every \(x\in \mathbb{R}\), the function \(t\mapsto u(t,x)\) is continuously differentiable. This completes the proof of Lemma \ref{handle}.
\end{proof}

 \noindent{\bf Proof of  Theorem \ref{globalexistence1}.} Assume $u_0 \in H^2(\R)$ and $u_1 \in H^1(\R)$, and let $p>1+\frac{2}{\alpha_0}$. Fix a small parameter $0<\varepsilon_0\ll1$ such that  
$\|u_0\|_{H^2}+\|u_1\|_{H^1}\leq\varepsilon_0$.
 For $T>0$, define the energy solution space $X(T)=   \mathcal{C}([0,T],H^2(\mathbb{R}))\cap \mathcal{C}^1([0,T],H^1(\mathbb{R}))\cap \mathcal{C}^2([0,T],L^2(\mathbb{R}))$,
equipped with the norm
$$\|v\|_{X(T)}=\sup_{0\leq t\leq T}(1+t)^{\frac{\alpha}{2}}\left\{ (1+t)^{-1}\| v(t)\|_{L^2}+\| \partial_x v(t)\|_{H^1}+\|\partial_t v(t)\|_{H^1}+\|\partial^2_{t} v(t)\|_{L^2}\right\},$$
where $\alpha$ is defined in \eqref{alpha}. We are going to use the Banach fixed-point theorem. Let us define the following complete metric space $B_{M}(T)=\{v\in X(T);\,\,\|v\|_{X(T)}\leq 2 M\},$ endowed with the distance induced by the norm $\|\cdotp\|_{X(T)}$, where $M>0$ is a positive constant to be chosen later. By Proposition \ref{prop2.1} and estimates \eqref{An1}, \eqref{Bn1}, and \eqref{Bn2} below, define the mapping $\Phi:B_M(T)\rightarrow X(T)$ by
$$\Phi(u)(t)=R(t)(u_0,u_1)+\int_0^tS(t,s)f(\partial_tu)\,ds=:\Phi(u)^{lin}(t)+\Phi(u)^{nl}(t),\quad\hbox{for $u\in B_M(T)$}.$$
Our goal is to prove that $\Phi$  is a contraction mapping on $B_M(T)$.\\

{\bf Step 1. $\Phi:B_M(T)\longrightarrow B_M(T)$.} Let $u\in B_M(T)$.\\
\noindent{\it Estimation of $(1+t)^{-1} \|\Phi(u)(t)\|_{L^2}+ \|\partial_t\Phi(u)(t)\|_{L^2}+\|\partial_x\Phi(u)(t)\|_{L^2}$}. Using  \eqref{0A1} in the particular case $s_0=0$, we have
\begin{equation} \label{GG1}
(1+t)^{-1}
\|\Phi(u)^{lin}(t)\|_{L^2}+
\|\partial_x \Phi(u)^{lin}(t)\|_{L^2}+
\|\partial_t \Phi(u)^{lin}(t)\|_{L^2}
	\leq C \,\varepsilon_0
 ({1+t})^{-\frac{\alpha}2}.
\end{equation}
Moreover,  by applying inequality \eqref{0A1} once again with $s_0=s$, and taking the initial data as $(0, f(\partial_tu(s)))$, we obtain
$$(1+t)^{-1}
\| \Phi(u)^{nl}(t)\|_{L^2}+
\|\partial_t \Phi(u)^{nl}(t)\|_{L^2}+\|\partial_x  \Phi(u)^{nl}(t)\|_{L^2}\leq C\,\int_0^t\left(\frac{1+s}{1+t}\right)^{\frac{\alpha}{2}}\|f(\partial_tu(s))\|_{L^2}\,ds.
$$
By the Sobolev embedding $H^1(\mathbb{R})\hookrightarrow L^\infty(\mathbb{R})$,  and the fact that $u\in B_M(T)$, we infer 
\begin{align}\label{S01}
\|\partial_tu(s)\|_{L^{\infty}}\leqslant C \|\partial_tu(s)\|_{H^1}\leqslant  C(1+s)^{-\frac{\alpha}{2}}\|u\|_{X(T)}\le C M(1+s)^{-\frac{\alpha}{2}}.
\end{align}
Therefore, by \eqref{S01} and  exploiting again  the fact that $u\in B_M(T)$, we conclude
\begin{equation}\label{An1}
\|f(\partial_t u(s))\|_{L^2}\leq \|\partial_tu(s)\|_{\infty}^{p-1}\|\partial_tu(s)\|_{L^2}\leq 
C  M^p(1+s)^{-\frac{\alpha p}2}.
\end{equation}
Then, we obtain
$$
(1+t)^{-1}\| \Phi(u)^{nl}(t)\|_{L^2}+
\|\partial_t \Phi(u)^{nl}(t)\|_{L^2}+\|\partial_x  \Phi(u)^{nl}(t)\|_{L^2}\leq C\,M^p(1+t)^{-\frac{\alpha}{2}}\int_0^t(1+s)^{-\frac{(p-1)\alpha}{2}}\,ds.
$$
Since $p>1+\frac{2}{\alpha_0}$, it follows that $\frac{2}{p-1}<\alpha_0$. Hence, if $\mu \in (\frac{2}{p-1},\mu_0)$,   in the case $\mu_0\in (0,1]$,  we easily deduce that
\begin{equation}\label{alpha-alpha0}
 \frac{(p-1)\alpha}{2}>1.
\end{equation}
Therefore, the integral is uniformly bounded, and we get
\begin{equation}\label{G200}
(1+t)^{-1}\| \Phi(u)^{nl}(t)\|_{L^2}+
\|\partial_t \Phi(u)^{nl}(t)\|_{L^2}+\|\partial_x  \Phi(u)^{nl}(t)\|_{L^2}\leq
 C\,M^p(1+t)^{-\frac{\alpha}{2}}.
\end{equation}
Combining \eqref{GG1} and \eqref{G200}, we conclude
\begin{equation}\label{G2T}
(1+t)^{-1}\| \Phi(u)(t)\|_{L^2}+
\|\partial_t \Phi(u)(t)\|_{L^2}+\|\partial_x  \Phi(u)(t)\|_{L^2}\leq  C \,\varepsilon_0
 ({1+t})^{-\frac{\alpha}2}+
 C\,M^p(1+t)^{-\frac{\alpha}{2}}.
\end{equation}

\noindent{\it Estimation of $\|\partial^2_{xt}\Phi(u)(t)\|_{L^2}+\|\partial_t^2 \Phi(u)(t)\|_{L^2}{+\|\partial_x^2 \Phi(u)(t)\|_{L^2}}$.} By \eqref{5mai1} with $s_0=0$, we have
\begin{equation}\label{H6} 
{\|\partial_x^2 \Phi(u)^{lin}(t)\|_{L^2}+}\|\partial^2_{xt} \Phi(u)^{lin}(t)\|_{L^2}+\|\partial_t^2 \Phi(u)^{lin}(t)\|_{L^2}
	\leq C \,\varepsilon_0
 ({1+t})^{-\frac{\alpha}2}.
\end{equation}
Moreover, applying estimate \eqref{5mai1} again with $s_0=s$, and considering the initial data $(0, f(\partial_tu(s)))$, we obtain
\begin{eqnarray}\label{fh}
&{}&\|\partial_x^2 \Phi(u)^{nl}(t)\|_{L^2}+\|\partial^2_{xt}\Phi(u)^{nl}(t)\|_{L^2}+\|\partial_t^2 \Phi(u)^{nl}(t)\|_{L^2}\nonumber\\
&{}&\leq C\,\int_0^t\left(\frac{1+s}{1+t}\right)^{\frac{\alpha}{2}}\left(\| f(\partial_tu(s))\|_{H^1}+\|\partial_t f(\partial_tu(s))\|_{L^2}\right)\,ds.
\end{eqnarray}
Based on \eqref{S01} and exploiting again  the fact that $u\in B_M(T)$, we have
\begin{equation}\label{Bn1}
\|\partial_x f(\partial_tu(s))\|_{L^2}=\|p|\partial_tu|^{p-1}\partial^2_{xt} u\|_{L^2}\leq p\|\partial_tu\|_\infty^{p-1}\|\partial^2_{xt}u\|_{L^2}\leq C  M^p(1+s)^{-\frac{p\alpha}{2}},
\end{equation}
and
\begin{equation}\label{Bn2}
\|\partial_t f(\partial_tu(s))\|_{L^2}=\|p|\partial_tu|^{p-1} \partial^2_tu\|_{L^2}\leq p\|\partial_tu\|_\infty^{p-1}\|\partial^2_tu\|_{L^2}\leq C  M^p(1+s)^{-\frac{p\alpha}{2}}.
\end{equation}
Hence, combining \eqref{An1}, \eqref{Bn1}, \eqref{Bn2}  and the estimate \eqref{fh}, we infer
\begin{equation} 
{\|\partial_x^2 \Phi(u)^{nl}(t)\|_{L^2}+}\|\partial^2_{xt} \Phi(u)^{nl}(t)\|_{L^2}+\|\partial_t^2 \Phi(u)^{nl}(t)\|_{L^2}
\leq C\,M^p(1+t)^{-\frac{\alpha}{2}}\int_0^t(1+s)^{-\frac{(p-1)\alpha}{2}}\,ds.
\end{equation}
It follows from condition \eqref{alpha-alpha0} that
\begin{equation} \label{7mai3}
\|\partial_x^2 \Phi(u)^{nl}(t)\|_{L^2}+\|\partial^2_{xt}\Phi(u)^{nl}(t)\|_{L^2}+\|\partial_t^2 \Phi(u)^{nl}(t)\|_{L^2}
\leq C\,M^p(1+t)^{-\frac{\alpha}{2}}.
\end{equation}
Therefore, by \eqref{H6} and \eqref{7mai3}, we get
\begin{equation} \label{G2000}
\|\partial_x^2 \Phi(u)(t)\|_{L^2}+\|\partial^2_{xt} \Phi(u)(t)\|_{L^2}+\|\partial_t^2 \Phi(u)(t)\|_{L^2}
\leq C\,\varepsilon_0 (1+t)^{-\frac{\alpha}{2}}
+ C\,M^p(1+t)^{-\frac{\alpha}{2}}.
\end{equation}
Summing up the  estimates \eqref{G2T} and \eqref{G2000}, we conclude that $ \|\Phi(u)\|_{X(T)}\leq  C\,\varepsilon_0+C\,M^p$. By choosing $M>0$ such that $C\,M^{p-1}\leq \frac12$, and then taking $0<\varepsilon_0\ll1$ small enough to ensure $C\,\varepsilon_0 \le \frac12M$, we arrive at $ \|\Phi(u)\|_{X(T)}\leq M$, that is, $\Phi(u)\in B_M(T)$.\\

{\bf Step 2. $\Phi$ is a contraction.} Let $u,v\in B_M(T)$.\\
\noindent{\it Estimation of $(1+t)^{-1}\|(\Phi(u)-\Phi(v))(t)\|_{L^2}+
\|\partial_t(\Phi(u)-\Phi(v))(t)\|_{L^2}+\|\partial_x(\Phi(u)-\Phi(v))(t)\|_{L^2}$.} Applying  \eqref{0A1} with $s_0=s$, and initial data as $(0, f(\partial_tu(s))-f(\partial_tv(s)))$, we derive
\begin{multline}
(1+t)^{-1}
\| \Phi(u)(t)-\Phi(v)(t)
\|_{L^2}+
\|\partial_t \Phi(u)(t)-\partial_t \Phi(v)(t)\|_{L^2}+\|\partial_x  \Phi(u)(t)-\partial_x  \Phi(v)(t)\|_{L^2}\\
\leq C\,\int_0^t\left(\frac{1+s}{1+t}\right)^{\frac{\alpha}{2}}\|f(\partial_tu(s))-f(\partial_tv(s))\|_{L^2}\,ds.
\end{multline}
By the basic inequalities \eqref{ab}-\eqref{abb}, we have
$$\|f(\partial_tu(s))-f(\partial_tv(s))\|_{L^2}\leq \|\partial_tu(s)-\partial_tv(s)\|_{L^2}\left(\|\partial_tu(s)\|^{p-1}_{\infty}+\|\partial_tv(s)\|^{p-1}_{\infty}\right).$$
Note that
$$
\|\partial_tu(s)-\partial_tv(s)\|_{L^2}\leq (1+s)^{-\frac{\alpha}{2}}\|u-v\|_{X(T)}.
$$
Using the Sobolev embedding $H^1(\mathbb{R})\hookrightarrow L^\infty(\mathbb{R})$, it follows that
$$\|\partial_tu(s)\|_{\infty}^{p-1}\leq C\|\partial_tu(s)\|_{H^1}^{p-1}\leq C (1+s)^{-\frac{\alpha (p-1)}{2}}\|u\|_{X(T)}^{p-1}
\leq C M^{p-1}(1+s)^{-\frac{\alpha (p-1)}{2}},$$
and similarly for $\partial_t v(s)$. Accordingly,
\begin{equation}\label{Fn1}
\|f(\partial_tu(s))-f(\partial_tv(s))\|_{L^2}\leq CM^{p-1}(1+s)^{-\frac{p\alpha }{2}}\|u-v\|_{X(T)}.
\end{equation}
Hence,
\begin{multline*}
(1+t)^{-1}
\| \Phi(u)(t)-\Phi(v)(t)
\|_{L^2}+
\|\partial_t \Phi(u)(t)-\partial_t \Phi(v)(t)\|_{L^2}+\|\partial_x  \Phi(u)(t)-\partial_x  \Phi(v)(t)\|_{L^2}\\
\leq C\,M^{p-1}\|u-v\|_{X(T)}
(1+t)^{-\frac{\alpha }{2}}\int_0^t(1+s)^{-\frac{(p-1)\alpha }{2}}\,ds.
\end{multline*}
Referring to \eqref{alpha-alpha0}, we deduce
\begin{multline*}
(1+t)^{-1}
\| \Phi(u)(t)-\Phi(v)(t)
\|_{L^2}+
\|\partial_t \Phi(u)(t)-\partial_t \Phi(v)(t)\|_{L^2}+\|\partial_x  \Phi(u)(t)-\partial_x  \Phi(v)(t)\|_{L^2}\\
\leq C\,M^{p-1}\|u-v\|_{X(T)}
(1+t)^{-\frac{\alpha}{2}}.
\end{multline*}
\noindent{\it Estimation of $\|\partial_t^2\left(\Phi(u)-\Phi(v)\right)(t)\|_{L^2}+\|\partial^2_{xt}\left(\Phi(u)-\Phi(v)\right)(t)\|_{L^2}+\|\partial_x^2\left(\Phi(u)-\Phi(v)\right)(t)\|_{L^2}$.} Applying \eqref{5mai1},  we obtain
\begin{eqnarray}
&{}&\|\partial_t^2\left(\Phi(u)-\Phi(v)\right)(t)\|_{L^2}+\|\partial^2_{xt}\left(\Phi(u)-\Phi(v)\right)(t)\|_{L^2}+\|\partial_x^2\left(\Phi(u)-\Phi(v)\right)(t)\|_{L^2}\nonumber\\
&{}&\leq\, C\,\int_0^t\left(\frac{1+s}{1+t}\right)^{\frac{\alpha }{2}}\left(\|f(\partial_tu)-f(\partial_tv)\|_{H^1}+\|\partial_t \left(f(\partial_tu)-f(\partial_tv)\right)\|_{L^2}\right)\,ds.\label{Hn1}
\end{eqnarray}
Since $p>1+\frac{2}{\alpha_0}\geq 3$, it follows that  \eqref{21B}-\eqref{21BG} can be applied. Therefore
\begin{align*}
&\left\|\partial_x \left(f(\partial_tu(s))-f(\partial_tv(s))\right)\right\|_{L^2}\\
&\leqslant C\|\partial^2_{xt} w(s)\|_{L^{2}}\|\partial_tu(s)\|^{p-1}_{L^{\infty}}+C\|\partial_tw(s)\|_{L^{\infty}}\left(\| \partial_tu(s)\|^{p-2}_{L^{\infty}}+\| \partial_tv(s)\|^{p-2}_{L^{\infty}}\right)\|\partial^2_{xt}v(s)\|_{L^{2}},
\end{align*}
where $w(t,x)\doteq u(t,x)-v(t,x)$. Using the Sobolev embedding again, $H^1(\mathbb{R})\hookrightarrow L^\infty(\mathbb{R})$,
$$\|\partial_tu(s)\|_{L^{\infty}}\leqslant CM(1+s)^{-\frac{\alpha }{2}},\quad \|\partial_tv(s)\|_{L^{\infty}}\leqslant CM(1+s)^{-\frac{\alpha }{2}},\quad \|\partial_tw(s)\|_{L^{\infty}}\leqslant C(1+s)^{-\frac{\alpha }{2}}\|u-v\|_{X(T)},$$
and
$$\|\partial^2_{xt}v(s)\|_{L^{2}}\leqslant CM(1+s)^{-\frac{\alpha }{2}},\qquad
\|\partial^2_{xt} w(s)\|_{L^{2}}\leqslant (1+s)^{-\frac{\alpha }{2}}\|u-v\|_{X(T)}.$$
Whence,
\begin{equation}\label{Gn1}\left\|\partial_x \left(f(\partial_tu(s))-f(\partial_tv(s))\right)\right\|_{L^2}\leqslant CM^{p-1}(1+s)^{-\frac{p\alpha }{2}}\|u-v\|_{X(T)}.
\end{equation}
Similarly, by \eqref{21C}-\eqref{21BC},
\begin{equation}\label{Gn10}
\left\|\partial_t \left(f(\partial_tu(s))-f(\partial_tv(s))\right)\right\|_{L^2}\leqslant CM^{p-1}(1+s)^{-\frac{p\alpha }{2}}\|u-v\|_{X(T)}.
\end{equation}
Inserting (\ref{Fn1}), (\ref{Gn1}), and (\ref{Gn10}) into (\ref{Hn1}) and taking into account \eqref{alpha-alpha0}, we derive
\begin{multline*}
\|\partial_t^2\left(\Phi(u)-\Phi(v)\right)(t)\|_{L^2}+\|\partial^2_{xt}\left(\Phi(u)-\Phi(v)\right)(t)\|_{L^2}+\|\partial_x^2\left(\Phi(u)-\Phi(v)\right)(t)\|_{L^2}\\
\leq CM^{p-1}\|u-v\|_{X(T) }(1+t)^{-\frac{\alpha }{2}}.
\end{multline*}
Summing all estimates, it follows that $ \|\Phi(u)-\Phi(v)\|_{X(T)}\leq  C\,M^{p-1}\|u-v\|_{X(T)}$. Choosing $M>0$ so that $C\,M^{p-1}\leq 1/2$, we arrive at
$$ \|\Phi(u)-\Phi(v)\|_{X(T)}\leq \frac{1}{2}\|u-v\|_{X(T)}.$$
By the Banach fixed-point theorem, the proof is complete.
\hfill$\square$


\section{Proof of Theorem \ref{globalexistence2}}\label{sec4}
Since the proof of Theorem \ref{globalexistence2} follows a similar approach to that of Theorem \ref{globalexistence1}, we outline the key steps while highlighting the necessary modifications. \\
\noindent {\bf Proof of Theorem \ref{globalexistence2}-(i).} To show that $\Phi:B_M(T)\longrightarrow B_M(T)$, we proceed in a similar manner as in the proof of Theorem \ref{globalexistence1}, utilizing the estimates
$$
\|f(\partial_t u(s),\partial_x u(s))\|_{L^2}\leq \|\partial_x u(s)\|_{\infty}^{q-1}\|\partial_x u(s)\|_{L^2}\leq C\,\|\partial_x u(s)\|_{H^1}^{q}\leq 
C  M^p(1+s)^{-\frac{q\alpha  }2},
$$
$$
\|\partial_x f(\partial_t u(s),\partial_x u(s))\|_{L^2}=\|\partial_x |\partial_x u(s)|^q\|_{L^2}\leq q\|\partial_x u\|_\infty^{q-1}\|\partial_x^2 u\|_{L^2}\leq C\,\|\partial_x u\|_{H^1}^{q}\leq C  M^p(1+s)^{-\frac{q\alpha }{2}},
$$
and
$$
\|\partial_t f(\partial_t u(s),\partial_x u(s))\|_{L^2}=\|\partial_t |\partial_x u(s)|^q\|_{L^2}\leq q\|\partial_x u\|_\infty^{q-1}\|\partial^2_{xt} u\|_{L^2}\leq C\,\|\partial_x u\|_{H^1}^{q-1}\|\partial^2_{xt} u\|_{L^2}\leq C  M^p(1+s)^{-\frac{q\alpha }{2}}.
$$
For the contraction mapping argument, we require the estimates:
\begin{eqnarray*}
\|f(\partial_t u(s),\partial_x u(s))-f(\partial_t v(s),\partial_x v(s))\|_{L^2}
&{}&\leq \|\partial_x u(s)-\partial_x v(s)\|_{L^2}\left(\|\partial_x u(s)\|^{q-1}_{\infty}+\|\partial_x v(s)\|^{q-1}_{\infty}\right)\\
&{}&\leq C\,(1+s)^{-\frac{\alpha }{2}} \|u- v\|_{X(T)}\left(\|\partial_x u(s)\|^{q-1}_{H^1}+\|\partial_x v(s)\|^{q-1}_{H^1}\right)\\
&{}&\leq C\,M^{q-1}(1+s)^{-\frac{q\alpha }{2}} \|u- v\|_{X(T)},
\end{eqnarray*}
\begin{eqnarray*}
&{}&\|\partial_x\left(f(\partial_t u(s),\partial_x u(s))-f(\partial_t v(s),\partial_x v(s))\right)\|_{L^2}\\
&{}&\leq C\|\partial_x^2 w(s)\|_{L^{2}}\|\partial_x u(s)\|^{q-1}_{L^{\infty}}+C\|\partial_x w(s)\|_{L^{\infty}}\left(\| \partial_x u(s)\|^{q-2}_{L^{\infty}}+\| \partial_x v(s)\|^{q-2}_{L^{\infty}}\right)\|\partial_x^2 v(s)\|_{L^{2}}\\
&{}&\leq C\,M^{q-1}(1+s)^{-\frac{q\alpha }{2}} \|u- v\|_{X(T)},
\end{eqnarray*}
and
\begin{eqnarray*}
&{}&\|\partial_t\left(f(\partial_t u(s),\partial_x u(s))-f(\partial_t v(s),\partial_x v(s))\right)\|_{L^2}\\
&{}&\leq C\|\partial^2_{xt} w(s)\|_{L^{2}}\|\partial_x u(s)\|^{q-1}_{L^{\infty}}+C\|\partial_t w(s)\|_{L^{\infty}}\left(\| \partial_x u(s)\|^{q-2}_{L^{\infty}}+\| \partial_x v(s)\|^{q-2}_{L^{\infty}}\right)\|\partial^2_{xt} v(s)\|_{L^{2}}\\
&{}&\leq C\,M^{q-1}(1+s)^{-\frac{q\alpha }{2}} \|u- v\|_{X(T)},
\end{eqnarray*}
where $w(t,x)\doteq u(t,x)-v(t,x)$. These estimates ensure the existence and uniqueness of solutions via the Banach fixed-point theorem.
\hfill$\square$\\

\noindent {\bf Proof of Theorem \ref{globalexistence2}-(ii).} We adopt a similar approach as in the proof of Theorem \ref{globalexistence1}, emplying the reasoning established in the proof of Theorem \ref{globalexistence2}-(i) along with the estimate
\begin{eqnarray*}
\|f(\partial_t u(s),\partial_x u(s))\|_{L^2}&\leq& \|\partial_t u(s)\|^p_{L^\infty}\||\partial_x u(s)|^q\|_{L^2}\\
&\leq& C\,\|\partial_t u(s)\|^p_{H^1}\||\partial_x u(s)|^q\|_{L^2}\\
&\leq& C\,M^{p}(1+s)^{-\frac{p\alpha }{2}}M^{q}(1+s)^{-\frac{q\alpha }{2}}\\
&=& C\,M^{p+q}(1+s)^{-\frac{\alpha (p+q)}{2}}.
\end{eqnarray*}
Therefore,
$$
(1+t)^{-1}\| \Phi(u)^{nl}(t)\|_{L^2}+
\|\partial_t \Phi(u)^{nl}(t)\|_{L^2}+\|\partial_x  \Phi(u)^{nl}(t)\|_{L^2}\leq C\,M^{p+q}(1+t)^{-\frac{\alpha }{2}}\int_0^t(1+s)^{-\frac{(p+q-1)\alpha }{2}}\,ds.
$$
Since $p+q>1+\frac{2}{\alpha_0 }$, the integral converges and we conclude
$$
(1+t)^{-1}\| \Phi(u)^{nl}(t)\|_{L^2}+
\|\partial_t \Phi(u)^{nl}(t)\|_{L^2}+\|\partial_x  \Phi(u)^{nl}(t)\|_{L^2}\leq
 C\,M^{p+q}(1+t)^{-\frac{\alpha }{2}}.
$$
This ensures that $\Phi$ maps $B_M(T)$ into itself, and a similar argument establishes the contraction property, thereby concluding the proof.
\hfill$\square$

\section{Proof of Theorem \ref{th2}}\label{sec5}


This section is dedicated to the proof of Theorem \ref{th2}, which aims to determine the blow-up region and establish lifespan estimates for problem \eqref{NLWb} in the one-dimensional space.\\
We first introduce a positive test function which is a particular positive  solution
 $\psi(x, t)$  with separated variables and satisfies the conjugate equation corresponding to  the linear problem, namely $\psi(x, t)$  satisfies
\begin{equation}\label{lambda-eq}
\partial^2_t \psi(x, t)-\partial_x^2 \psi(x, t) -V(x)\partial_t\psi(x, t)=0.
\end{equation}
More precisely, we choose the function $\psi$ given by
\begin{equation}
\label{psi33}
\psi(x,t):=\rho(t)\phi(x);
\quad
\rho(t):=e^{-t},
\end{equation}
where $\phi(x)=\phi(-x)$  is  a solution of the  problem
\begin{equation}\label{phi}
 \phi''(x) =\left(1+V(x)\right)\phi(x), \qquad \hbox{for all}\, \, x\in \R,
\end{equation}
with initial data $\phi(0)=1$, $\phi'(0)=0$. Indeed, $\phi$ can be obtained as a solution of the integral equation
$$\phi(r)=1+\int_0^r(r-s)(1+V(s))\phi(s)\,ds,\qquad r>0.$$
Note that the existence of a positive solution for the associated problem of \eqref{phi} is studied in \cite{Ikehata,LLTW} in higher dimension.
However, the specific case of one spatial dimension has not been addressed in these works. 
However, following the strategy used in \cite{Ikehata} in the case of  higher dimension, 
we easily prove the following properties of $\phi$.
\begin{lemma}
\label{lem0}
  There exists an even function $\phi$ solution of \eqref{phi} verifying $\phi(0)=1$, $\phi'(0)=0$,
which satisfies
\begin{equation}
\label{elip1}
0<\phi (x)\le C (1+|x|)^{\frac{\mu_0}{2}}e^{|x|}, \qquad \hbox{for all}\,\, \,x\in \R.
\end{equation}
\end{lemma}

\begin{proof}
Clearly, we have $\phi(x)>0$, and  $x\phi'(x)\ge0$ for all $x\in \R$. Let $\phi(x)=e^{x+q(x)}$. 
By substituting $\phi$ into \eqref{phi}, one can easily see that the function $q$ satisfies
\begin{equation}\label{functionq}
q''(x)+2q'(x)+\big(q'(x)\big)^2=V(x), \qquad \hbox{for all}\,\, \,x\in \R,
\end{equation}
where $q(0)=0$ and $q'(0)=-1$.
Multiplying \eqref{functionq} by $e^{2x}$, we deduce after integrating over $[0, x]$ that
\begin{equation}\label{functionq1}
e^{2x}q'(x)\le -1+\int_0^xe^{2y}V(y)dy, \qquad \hbox{for all}\,\, \, x>0.
\end{equation}
By using the expression of $V$,  and integrating by parts we infer 
\begin{equation}\label{functionq2}
\int_0^xe^{2y}V(y)dy=\frac{\mu_0}{2\sqrt{1+x^2}}e^{2x}-\frac12\int_0^xe^{2y}V'(y)dy, \qquad \hbox{for all}\,\, \, x>0,
\end{equation}
which is asymptotically
$$\frac{\mu_0}{2}x^{-1}e^{2x}+\mathcal{O}(x^{-2})e^{2x},$$
since $V(x)\sim\mu_0 x^{-1}$ and $V'(x)\sim-\mu_0 x^{-2}$ for large $x$. Hence, by \eqref{functionq1} and \eqref{functionq2}, we write for $x$ large enough,
\begin{equation}\label{functionq3}
q'(x)\le \frac{\mu_0}{2x}+\mathcal{O}(x^{-2}).
\end{equation}
Therefore,
 for $x$ large enough,
\begin{equation}\label{functionq4}
q(x)\le \frac{\mu_0}{2}\ln (x)+\mathcal{O}(1).
\end{equation}
Inequality \eqref{elip1} is a direct consequence of \eqref{functionq4} owing to the positivity and the symmetry of  $\phi$.
\end{proof}

\par

Furthermore, by employing  the estimate \eqref{elip1}  along with the expression of $\psi$ given by \eqref{psi33}, and  proceeding along the same lines as in the derivation of inequality (2.5) in \cite{YZ06} or \cite[Lemma~3.1]{HF}, we obtain the following result. 

\begin{lemma} 
\label{lem1}
There exists a constant $C=C(R_0,\mu_0)>0$ such that
\begin{equation}
\label{psi}
\int_{|x|\leq R_0+t}\psi(x,t)dx
\leq C
 (1+t)^{\frac{\mu_0 }{2}},
\quad\hbox{for all}\,\, t\ge0.
\end{equation}
\end{lemma}
The estimate \eqref{psi} plays a crucial role in establishing the blow-up result stated in Theorem \ref{th2}. \\

 \noindent{\bf Proof of  Theorem \ref{th2}.} The proof relies on Lemma \ref{lem1} and follows a similar strategy to the one used in \cite[Theorem~2.3]{HF}. Specifically, we employ the test function method, originally introduced by Baras and Pierre \cite{Baras}, later refined by Zhang \cite{Zhang2001}, and further developed in the works of Mitidieri and Pohozaev \cite{19}, which has since been widely applied in various contexts,  including the works of Kirane et al.  \cite{KLT} and Fino et al. \cite{Fino1,Fino3,Fino4, HF0}.
 \hfill$\square$


\bibliographystyle{elsarticle-num}


\end{document}